\documentclass[11pt,reqno,a4paper]{amsart}
\usepackage[T1]{fontenc}
\usepackage{microtype, a4wide, textcomp,fbb}
\usepackage{
    amsmath,  amssymb,  amsthm,   amscd,
    gensymb,  graphicx, etoolbox, 
    booktabs, stackrel, mathtools    
}

\usepackage[all]{xy}
\usepackage{epsfig,float}
\usepackage{tikz-cd}
\usetikzlibrary{positioning}
\usetikzlibrary{shapes,arrows.meta,calc}
\usepackage{tikz}
\usepackage{wrapfig}
\usepackage{caption}
\usepackage{subcaption}
\usepackage[colorlinks=true, linkcolor=blue, citecolor=blue, urlcolor=blue, breaklinks=true]{hyperref}
\usepackage[capitalise]{cleveref}
\usepackage{placeins}
\usepackage{relsize}
\usepackage{todonotes}
\usepackage[]{caption}

\usepackage{fullpage}
\usepackage{array}      
\newtheorem{theorem}{Theorem}[section]

\newtheorem{lemma}[theorem]{Lemma}
\newtheorem{claim}[theorem]{Claim}

\newtheorem{proposition}[theorem]{Proposition}
\theoremstyle{definition}
\newtheorem{definition}[theorem]{Definition}
\newtheorem{example}[theorem]{Example}

\newtheorem{remark}[theorem]{Remark}
\newtheorem{question}[theorem]{Question}

\newcommand\R{\mathbb{R}}
\newcommand\Z{\mathbb{Z}}
\newcommand{\I}{\mathbb{I}}
\newcommand{\X}{\mathbf{X}}

\newcommand{\N}{\mathcal{N}}
\newcommand{\bN}{\mathbb{N}}
\newcommand{\K}{\mathcal{K}}

\newcommand{\bs}{\mathbb{S}}
\newcommand{\mes}{\textrm{mes}}
\newcommand{\lk}{\mathrm{lk}}
\newcommand{\st}{\mathrm{st}}
\newcommand{\susp}{\Sigma\,}
\newcommand{\cU}{\mathcal{U}}

\usepackage{tikz}
\usetikzlibrary{arrows}
\usetikzlibrary{shapes}
\usepackage{float}
\usepackage{tabularx}
\usepackage{kantlipsum}
\usepackage{array}

\begin{document}

\title{{\v C}ech complexes of hypercube graphs}
\author{Henry Adams}
\address{University of Florida, Gainesville}
\email{henry.adams@ufl.edu}
\author{Samir Shukla}
\address{Indian Institute of Technology Mandi, India}
\email{samir@iitmandi.ac.in}
\author{Anurag Singh}
\address{Indian Institute of Technology Bhilai, India}
\email{anurags@iitbhilai.ac.in}
\date{\today}
\maketitle

\begin{abstract}
A \v{C}ech complex of a finite simple graph $G$ is a nerve complex of balls in the graph, with one ball centered at each vertex.
More precisely, the \v{C}ech complex $\N(G,r)$ is the nerve of all closed balls of radius $\frac{r}{2}$ centered at vertices of $G$, where these balls are drawn in the geometric realization of the graph $G$ (equipped with the shortest path metric).
The simplicial complex $\N(G,r)$ is equal to the graph $G$ when $r=1$, and homotopy equivalent to the graph $G$ when $r$ is smaller than half the length of the shortest cycle in $G$.
For higher values of $r$, the topology of $\N(G,r)$ is not well-understood.
We consider the $n$-dimensional hypercube graphs $\I_n$ with $2^n$ vertices.
Our main results are as follows.
First, when $r=2$, we show that the \v{C}ech complex $\N(\I_n,2)$ is homotopy equivalent to a wedge of 2-spheres for all $n\ge 1$, and we count the number of 2-spheres appearing in this wedge sum.
Second, when $r=3$, we show that $\N(\I_n,3)$ is homotopy equivalent to a simplicial complex of dimension at most 4, and that for $n\ge 4$ the reduced homology of $\N(\I_n, 3)$ is nonzero in dimensions 3 and 4, and zero in all other dimensions.
Finally, we show that for all $n\ge 1$ and $r\ge 0$, the inclusion $\N(\I_n, r)\hookrightarrow \N(\I_n, r+2)$ is null-homotopic, providing a bound on the length of bars in the persistent homology of \v{C}ech complexes of hypercube graphs.
\end{abstract}
\medskip



\medskip 

\section{Introduction}

Let $G=(V(G),E(G))$ be a finite simple connected graph.
We can equip the vertex set $V(G)$ with the shortest path metric, which extends to give a shortest path metric on the geometric realization of $G$, in which the realization of each edge is isometric to the unit interval $[0,1]$.
The \v{C}ech simplicial complex $\N(G,r)$ is the nerve of all closed balls of radius $\frac{r}{2}$ centered at vertices of $G$, where these balls are drawn in the geometric realization of the graph $G$; see Figure~\ref{fig:balls}.
In other words, the vertex set of $\N(G,r)$ is $V(G)$, and a set of vertices forms a simplex if their corresponding balls have a point of intersection.

\begin{figure}[htb]
\centering
\includegraphics[width=5in]{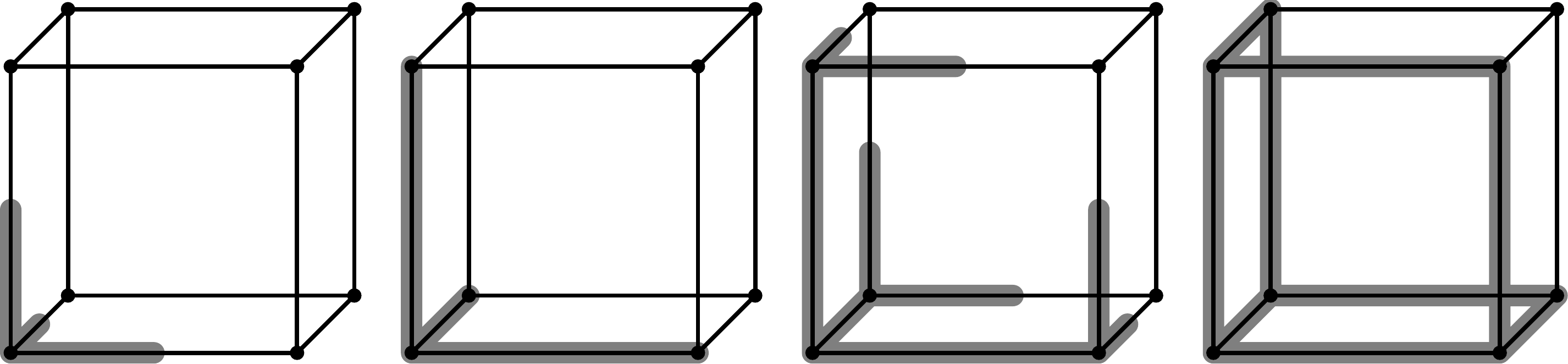}
\caption{Balls of radius $\frac{r}{2}$ centered at a vertex $v$ in the geometric realization of the 3-dimensional hypercube graph, for $r=1, 2, 3, 4$.}
\label{fig:balls}
\end{figure}

When $r=0$, the \v{C}ech simplicial complex $\N(G,0)$ is $V(G)$, a disjoint union of vertices.
When $r=1$, the complex $\N(G,1)$ is equal to the graph $G$.
If $\ell$ is the length of the shortest cycle in the graph $G$, and if $1\le r<\frac{\ell}{2}$, then any intersection of balls of radius $\frac{r}{2}$ is either empty or contractible, and hence the nerve lemma applies to guarantee that $\N(G,r)$ is homotopy equivalent to the geometric realization of $G$ (\Cref{lem:r-small-nerve}).
But for larger values of $r$, the hypotheses of the nerve lemma are no longer satisfied, and the topology of $\N(G,r)$ is not well-understood.

In applied and computational topology, \v{C}ech complexes are frequently-used tools to approximate the ``shape'' of a dataset~\cite{Carlsson2009}.
In the context of this paper, the dataset is the vertex set of a graph, which is a common source of data when modeling social networks, transportation schedules, the routing of messages in a communication network, or the structure of molecules.
More generally, one is given a finite dataset $X$ sampled, perhaps with noise, from some underlying space $M$, which may be a manifold, or a stratified space, or a graph.
One would like to approximate topological information about $M$ using only the sampling $X$.
The idea of persistent homology~\cite{edelsbrunner2000topological,zomorodian2005computing,EdelsbrunnerHarer} is to ``thicken'' $X$ depending on some scale parameter $r\ge 0$, and to measure how the topology of these thickenings changes as the scale $r$ increases.
Example choices of thickenings are the Vietoris--Rips complexes~\cite{Vietoris27}, witness complexes~\cite{DeSilvaCarlsson}, alpha complexes~\cite{edelsbrunner1994three,akkiraju1995alpha}, or \v{C}ech complexes~\cite{EdelsbrunnerHarer} of $X$.
In the case of \v{C}ech complexes, there are multiple options: one defines the \v{C}ech complex as the nerve of all balls with centers in $X$, where those balls are drawn in either $X$, or in $M$, or in $\R^m$ (if $M$ happens to be contained in some Euclidean space $\R^m$).
There are a wide variety of reconstruction results for several different types of thickenings, which give theoretical or probabilistic guarantees for being able to recover the homology groups or the homotopy type of the underlying space $M$ from these thickenings of $X$~\cite{guibas2008reconstruction,Hausmann1995,Latschev2001,niyogi2008finding}.
The most standard such reconstruction result is the nerve lemma for \v{C}ech complexes~\cite{Alexandroff1928,bjorner,Borsuk1948,Hatcher}.
These reconstruction results, however, typically require the scale parameter $r$ to be sufficiently small compared to the (unknown) geometric properties of the underlying space $M$, such as curvature if $M$ is a manifold, or such as the reach if $M$ is embedded in Euclidean space.
In accordance with the idea of persistent homology, data science practitioners increase the scale $r\ge 0$ to values much larger than these regimes, in order to see which topological features persist.
There is hence a need to understand these thickenings when the aforementioned reconstruction results no longer apply, including the case of \v{C}ech complexes when the balls are large enough so that the nerve lemma no longer applies.

\v{C}ech complexes are closely related to Vietoris--Rips complexes, and so we provide a brief introduction to the similarities and differences between these constructions.
Because they are constructed as nerve complexes, \v{C}ech complexes sometimes satisfy the hypotheses of the nerve theorem~\cite{Borsuk1948}, in which case the \v{C}ech complex is homotopy equivalent to a union of the balls.
Though Vietoris--Rips complexes are less likely to satisfy a nerve theorem, they are still equipped with reconstruction guarantees when the scale parameter is not too large~\cite{Hausmann1995,Latschev2001}.
Furthermore, in part because Vietoris--Rips complexes are \emph{clique} or \emph{flag} complexes, their persistent homology can be efficiently computed~\cite{bauer2021ripser}.
In many ways \v{C}ech and Vietoris--Rips complexes have similar behavior.
Indeed, their filtrations are multiplicatively interleaved~\cite{Carlsson2009}, and they both satisfy stability results when the underlying dataset is perturbed by a controllable amount with respect to the Gromov--Hausdorff distance~\cite{ChazalDeSilvaOudot2014,chazal2009gromov}.
This means that one could study the \v{C}ech or Vietoris--Rips persistent homology of manifolds by studying the \v{C}ech or Vietoris--Rips complexes of graphs that approximate that manifold.
Indeed, this was the approach taken in~\cite{AA-VRS1} when studying the \v{C}ech and Vietoris--Rips persistent homology of the circle.
For $G$ a finite connected graph, it is in general not easy to produce a list of the maximal simplices in a Vietoris--Rips complex of $G$.
However, this is often easier to do for \v{C}ech complexes (see for example \Cref{lem:maximal-simplices} in the case of hypercube graphs), and this is one reason why we consider \v{C}ech complexes here.

The papers~\cite{carlsson2020persistent,adamaszek2022vietoris,shukla2022vietoris,feng2023vietoris,feng2023homotopy,adams2023lower} study the Vietoris--Rips complexes of hypercube graphs, uncovering some structure when the scale parameter is small or when the dimension of homology is small, though many questions remain unanswered.
Questions about the shape of Vietoris--Rips complexes of hypercube graphs originally arose from work by mathematical biologists who were applying topology in order to study genetic trees, medial recombination, and reticulate evolution~\cite{emmett2015quantifying,emmett2016topology,camara2016topological,chan2013topology}.
In this paper, we consider analogous questions about \v{C}ech complexes of hypercube graphs.
More is known about the Vietoris--Rips and \v{C}ech complexes of cycle graphs, which are always homotopy equivalent to either a single odd-dimensional sphere or to a wedge sum of even-dimensional spheres of the same dimension~\cite{Adamaszek2013,AAFPP-J}.
The persistent homology of Vietoris--Rips complexes of finite connected graphs is studied in~\cite{adams2022geometric} (see also~\cite{lim2020vietoris}), but less is known about the persistent homology of \v{C}ech complexes of graphs.
The papers~\cite{gasparovic2018complete,virk20201} study the 1-dimensional persistence of \v{C}ech and Vietoris--Rips complexes of a different flavor, which have an uncountable number of vertices (one for each point in a metric graph or geodesic space), instead of the simplicial complexes with a finite number of vertices that we study here.

In this paper we consider the \v{C}ech complexes of $n$-dimensional hypercube graphs $\I_n$ with $2^n$ vertices.
Our main results are as follows.
When $r=2$, we show in \Cref{thm:main_for_index_2} that for all $n\ge 1$, the \v{C}ech complex $\N(\I_n,2)$ is homotopy equivalent to a $2^{n-2}(n^2-3n+4)$-fold wedge sum of 2-spheres.
When $r=3$, we show in \Cref{thm:collapsibility} that $\N(\I_n,3)$ is homotopy equivalent to a simplicial complex of dimension at most 4, and we show in \Cref{thm:main3} that for $n\ge 4$ the reduced homology of $\N(\I_n, 3)$ is nonzero in dimensions 3 and 4, and zero in all other dimensions.
Finally, in \Cref{thm:persistence} we show that for all $n\ge 1$ and $r\ge 0$, the inclusion $\N(\I_n, r)\hookrightarrow \N(\I_n, r+2)$ is null-homotopic, providing a bound on the length of bars in the \v{C}ech persistent homology of hypercube graphs.

We begin with preliminaries in \Cref{sec:preliminaries}, and a description of the \v{C}ech simplicial complexes $\N(\I_n, r)$ of hypergraphs in \Cref{sec:nerve-hypercube}.
In \Cref{sec:r=2} we study the case $r=2$, and in \Cref{sec:r=3} we consider the case $r=3$.
We analyze the persistent homology of the filtration $\N(\I_n, -)$ in \Cref{sec:persistence}, and we conclude by sharing some open questions in \Cref{sec:conclusion}.

\section{Preliminaries}
\label{sec:preliminaries}

We give preliminaries on (nerve) simplicial complexes, graphs, \v{C}ech complexes of graphs, hypercube graphs, and results about homotopy types of simplicial complexes.

\subsection*{Simplicial complexes and nerves}

A simplicial complex $\K$ on a vertex set $V$ is a family of subsets of $V$, including all singletons, such that if $\sigma\in \K$ and $\tau \subseteq \sigma$, then $\tau\in \K$.
We identify a simplicial complex with its geometric realization, which is a topological space.

Let $\cU=\{U_i\}_{i\in I}$ be a cover of a topological space $X$.
The \emph{nerve simplicial complex} $\mathbf{N}(\cU) = \mathbf{N}(\{U_i\})$ has $I$ as its vertex set, and has a finite subset $\sigma\subseteq I$ as a simplex if $\cap_{i\in \sigma}U_i\neq\emptyset$.
If each $U_i$ is contractible, and if each nonempty intersection $\cap_{i\in \sigma}U_i$ is contractible, then we say that the cover $\cU$ is a \emph{good cover}.
The nerve theorem provides relatively mild point-set topology assumptions so that if $\cU$ is a good cover of $X$, then the nerve $\mathbf{N}(\cU)$ is homotopy equivalent to the space $X$.
This theorem applies, for example, if $\cU$ is an open cover of a paracompact space, or if $\cU$ is a cover of a simplicial complex by subcomplexes~\cite{Borsuk1948,Hatcher,bjorner} (see \Cref{thm:nerve}).

\subsection*{Graphs}

Let $G=(V(G),E(G))$ be  a finite simple connected graph.
If two vertices $v$ and $w$ are adjacent, then we denote this by $v \sim w$.
For $S \subseteq V(G)$, the {\it induced subgraph} $G[S]$ is a subgraph of $G$, with  vertex set $S$ and  edge set  
$E(G) \cap {S \choose 2}$.

For $v \in V(G)$, the \emph{(open) neighborhood} of $v$ is $N_G(v) = \{w ~ | ~ w \sim v\}$.
When the graph $G$ is clear from context, we often simplify notation and write $N(v)$.
We equip the vertex set $V(G)$ with the shortest path metric $d\colon V(G)\times V(G)\to \mathbb{R}$.

\begin{definition}
\label{defn:neighborhood}
For a vertex $v\in V(G)$ and $r\ge 0$, the \emph{(closed) $r$-neighborhood} about $v$ is
\[N_{G,r}[v]=\{w\in V(G)~|~d(v,w)\le r\}.\]
We let $N_{G}[v]\coloneqq N_{G,1}[v]$ denote the closed $1$-neighborhood.
When the graph $G$ is clear from context, we often simplify notation and write $N_r[v]$.
\end{definition}

We emphasize that $v$ is not an element of the open neighborhood $N_G(v)$, but $v$ is an element of the closed neighborhood $N_r[v]$.
Furthermore, $N_1[v]=N(v)\cup\{v\}$.

\subsection*{\v{C}ech complexes of graph}

For $G$ a finite simple connected graph and $r\ge 0$ a real number, the \emph{\v{C}ech simplicial complex $\N(G,r)$} is the nerve of all closed balls of radius $\frac{r}{2}$ centered at vertices of $G$.
These balls are drawn in the geometric realization of $G$ equipped with the shortest path metric, in which the realization of each edge is isometric to the unit interval $[0,1]$.

We note that $\N(G,r)=\N(G,\lfloor r\rfloor)$ for all real numbers $r\ge 0$.
Therefore, it suffices to consider only the \v{C}ech complexes when the scale $r$ is an integer; we give an equivalent definition of the \v{C}ech complex in this case.

\begin{lemma}
\label{lem:cech}
For $G$ a finite simple connected graph and integer $r\ge 0$, the \emph{\v{C}ech simplicial complex $\N(G,r)$} can equivalently be defined as follows.
The vertex set of $\N(G,r)$ is $V(G)$.
If $r$ is even then the simplices of $\N(G,r)$ are generated by the closed neighborhoods $N_{\frac{r}{2}}[v]$ for $v\in V(G)$.
If $r$ is odd then the simplices of $\N(G,r)$ are generated by the sets $N_{\frac{r-1}{2}}[v]\cup N_{\frac{r-1}{2}}[w]$ for $(v,w)\in E(G)$.
\end{lemma}

\begin{proof}
All balls in this proof are centered at vertices.
Note that if $r$ is even, then any collection of closed balls of radius $\frac{r}{2}$ with nonempty intersection also intersect at a vertex.
And, a collection of such balls intersect at the vertex $v$ if and only if their center vertices are all contained in the neighborhood $N_\frac{r}{2}[v]$.
Note that if $r$ is odd, then any collection of closed balls of radius $\frac{r}{2}$ with nonempty intersection also intersect at the midpoint of an edge (in the geometric realization of $G$).
And, a collection of such balls intersect at the midpoint of the edge $(v,w)$ if and only if their center vertices are each contained in the union of neighborhoods $N_{\frac{r-1}{2}}[v]\cup N_{\frac{r-1}{2}}[w]$.
\end{proof}

A benefit of Lemma~\ref{lem:cech} is that it gives an explicit description of simplices in the \v{C}ech simplicial complex of a graph: any simplex in $\N(G,r)$ is either a vertex or a subset of one of the generating simplices.

Two early papers on the \v{C}ech complexes of graphs include~\cite{Previte2014}, whose \emph{$D$-neighborhood complexes} recover $\N(G,r)$ when $r$ is even and $D=\{0,1,\ldots,\frac{r}{2}\}$, and~\cite{AAFPP-J} which gives a complete characterization of the homotopy types of $\N(G,r)$ for all $r$ when $G$ is a cycle graph.

The following lemma describes when the \v{C}ech complex $\N(G,r)$ recovers the graph $G$ up to homotopy type.

\begin{lemma}
\label{lem:r-small-nerve}
If $\ell$ is the length of the shortest cycle in the graph $G$, and if $1\le r<\frac{\ell}{2}$, then $\N(G,r)$ is homotopy equivalent to the geometric realization of $G$.
\end{lemma}

\begin{proof}
We use the description of the \v{C}ech simplicial complex $\N(G,r)$ as the nerve of all closed balls of radius $\frac{r}{2}$ centered at vertices of $G$, where these balls are drawn in the geometric realization of $G$.
If $\ell$ is the length of the shortest cycle in the graph $G$, and if $1\le r<\frac{\ell}{2}$, then any intersection of such balls of radius $\frac{r}{2}$ is either empty or contractible.
Hence the nerve lemma~\cite[Corollary~4G.3]{Hatcher} applies to give that  that $\N(G,r)$ is homotopy equivalent to the union of all such balls, which is equal to the geometric realization of $G$.
\end{proof}

\subsection*{Hypercube graphs}

The $n$-dimensional hypercube graph $\I_n$ has $2^n$ vertices, all binary strings of length $n$, with two vertices adjacent if their Hamming distance is one.

\begin{definition}
\label{defn:hypercube}
For a positive integer $n$, the {\it $n$-dimensional  hypercube graph}, denoted by $\I_n$, is a graph whose vertex set is $V(\I_n)= \{x_1 \ldots x_n \ | \  x_i \in \{0, 1\} \ \forall \ 1 \leq i \leq n\}$ and  any two vertices $x_1 \ldots x_n$ and $y_1 \ldots y_n$ are adjacent if and only if $\sum\limits_{i=1}^n |x_i - y_i| = 1$, {\it i.e.}, the binary strings corresponding to the two vertices differ in exactly one position.
\end{definition}

We now fix some notation, which we use in the rest of the article.
For a positive integer $n$, we denote the set $\{1, \ldots, n\}$ by $[n]$.
Let  $x = x_1 \ldots x_n \in V(\I_n)$.
For any $i \in [n]$, we also denote  $x_i$ by $x(i)$.
For a $ \{i_1, i_2, \ldots, i_k\}\subseteq [n]$, we let 
$x^{i_1, \ldots, i_k} \in V(\I_n)$ be defined by 
\begin{equation}\label{eq:superscript}
    x^{i_1, \ldots, i_k}(j)  = \begin{cases}
\ x(j) & \text{if} \  j \notin \{i_1, \ldots, i_k\},\\
1 & \text{if} \   j \in \{i_1, \ldots, i_k\}\text{ and }x(j)=0, \\
0 & \text{if} \   j \in \{i_1, \ldots, i_k\}\text{ and }x(j)=1.
\end{cases}
\end{equation}

Observe  that $N_{\I_n}(x) = \{x^i : i \in [n]\}$.

\subsection*{Simplicial complex homotopy types}

In order to describe the homotopy types of certain simplicial complexes, we will use the notions of wedge sums and suspensions of topological spaces, as well as the language of stars and links in simplicial complexes.

Given two topological spaces $X$ and $Y$, their wedge sum $X\vee Y$ is the space obtained by gluing $X$ and $Y$ together at a single point.
The homotopy type of $X\vee Y$ is independent of this choice of points if $X$ and $Y$ are connected CW complexes.
For $n\ge 1$, let $\vee_n X$ denote the $n$-fold wedge sum of $X$, namely $\vee_n X = \underbrace{X\vee \ldots \vee X}_n$.

Let $\K$ be  a simplicial complex. The \emph{star} of a vertex $v\in V(\K)$ is $\st_{\K}(v)=\{\sigma\in \K : \sigma \cup \{v\}\in \K\}$.
Note that the star is contractible since it is a cone with apex $v$.
The \emph{link} of a vertex $v$ is $\lk_{\K}(v)=\{\sigma\in \K : v\notin \sigma\text{ and }\sigma \cup \{v\}\in \K\}$.
The \emph{deletion} of a vertex $v$, denoted $\K\setminus v$, is the induced simplicial complex on vertex set $V\setminus v$; the simplices of $\K\setminus v$ are all those simplices $\sigma\in \K$ such that $v\notin \sigma$.
Using the following result, which is well-known to computational topologists, the homotopy type of a complex $\K$ can be computed using the link and deletion of a vertex in $\K$ (under given conditions).

\begin{lemma}[{\cite[Lemma~1]{adamaszek2022vietoris}}]
\label{lem:splitting}
Let $\K$ be a simplicial complex, and let $v\in \K$ be a vertex such that the inclusion $\iota\colon \lk_{\K}(v)\hookrightarrow \K\setminus v$ is a null-homotopy.
Then up to homotopy we have a splitting $\K\simeq (\K\setminus v) \vee \Sigma\ \lk_{\K}(v)$.
\end{lemma}

More generally, we also have the following result that will be used in this article to determine the homotopy type of a simplicial complex.
We recall that if $X$ is a topological space, then its suspension is the quotient space $\susp X = (X \times [0,1]) / (X\times\{0\}, X\times\{1\})$.

\begin{lemma}[{\cite[Remark 2.4]{SSA22}}]
\label{lem:complex_union}
Let the simplicial complex $\K=K_1\cup K_2$ be such that the inclusion maps $K_1\cap K_2 \hookrightarrow K_1$ and $K_1\cap K_2 \hookrightarrow K_2$ are null-homotopies.
Then $\K\simeq K_1 \vee K_2 \vee \susp (K_1\cap K_2)$.
\end{lemma}

We note that \Cref{lem:splitting} is a particular case of \Cref{lem:complex_union}.
Indeed, observe that $\K = (\K\setminus v)\cup \st_{\K}(v)$ and $(\K\setminus v)\cap \st_{\K}(v)= \lk_{\K}(v) \hookrightarrow \st_{\K}(v)$.
Moreover, the fact that $\st_{\K}(v)$ is contractible implies that this inclusion is a null-homotopy.

A topological space $X$ is said to be $k$-{\it connected} if every map from an $m$-dimensional sphere $\bs^m \to X$ can be extended to a map from the $(m+1)$-dimensional disk $D^{m+1} \to X$ for $m = 0, 1, \ldots , k$.
If $X$ is $k$-connected, then all its homotopy groups $\pi_i(X)$ are trivial for $0 \leq i \leq k$.

The following result, also known as a \emph{nerve theorem}, can be used to compute the homotopy type of a simplicial complex $\K$ when $\K$ is union of two or more complexes.

\begin{theorem}\cite[Theorem 10.6]{bjorner}
\label{thm:nerve}
Let $\Delta$ be a simplicial complex and let $(\Delta_i)_{i \in I}$ be a family of subcomplexes such that $\Delta = \bigcup\limits_{i \in I} \Delta_i$.
\begin{itemize}
	\item[(i)] \label{nerve1}Suppose every nonempty finite intersection $\Delta_{i_1} \cap \ldots \cap \Delta_{i_t}$ for $i_j \in I$ and $t \in \bN$ is contractible.
    Then $\Delta$ and the nerve $\mathbf{N}(\{\Delta_i\})$ are homotopy equivalent.
	\item[(ii)] Suppose every nonempty finite intersection $\Delta_{i_1} \cap \ldots \cap \Delta_{i_t}$ is $(k-t+1)$-connected.
    Then $\Delta$ is $k$-connected if and only if  $\mathbf{N}(\{\Delta_i\})$ is $k$-connected.
\end{itemize}
\end{theorem}

\section{Description of $\N(\I_n,r)$}
\label{sec:nerve-hypercube}

Recall that $\I_n$ is the hypercube graph with $2^n$ vertices.
In this section we describe the maximal simplices in the \v{C}ech simplicial complexes $\N(\I_n,r)$, before then describing what is known about the homotopy types and homology groups of these complexes.

The following lemma is a simple consequence of \Cref{lem:cech}, which gives an equivalent definition of \v{C}ech complexes of graphs; see Figure~\ref{fig:maximal}.

\begin{lemma}
\label{lem:maximal-simplices}
For $n \ge 1$ and $r \ge 0$, let $\sigma$ be a maximal simplex of the \v{C}ech complex $\N(\I_n,r)$.
Then, there exists $u\in V(\I_n)$ or $(v,w) \in E(\I_n)$ such that
\begin{equation*}
\sigma= \begin{cases}
N_{\frac{r}{2}}[u] & \text{ if } r \text{ is even},\\
N_{\frac{r-1}{2}}[v]\cup N_{\frac{r-1}{2}}[w] & \text{ if } r \text{ is odd}.
\end{cases}
\end{equation*}
\end{lemma}

\begin{figure}[htb]
\centering
\includegraphics[width=5in]{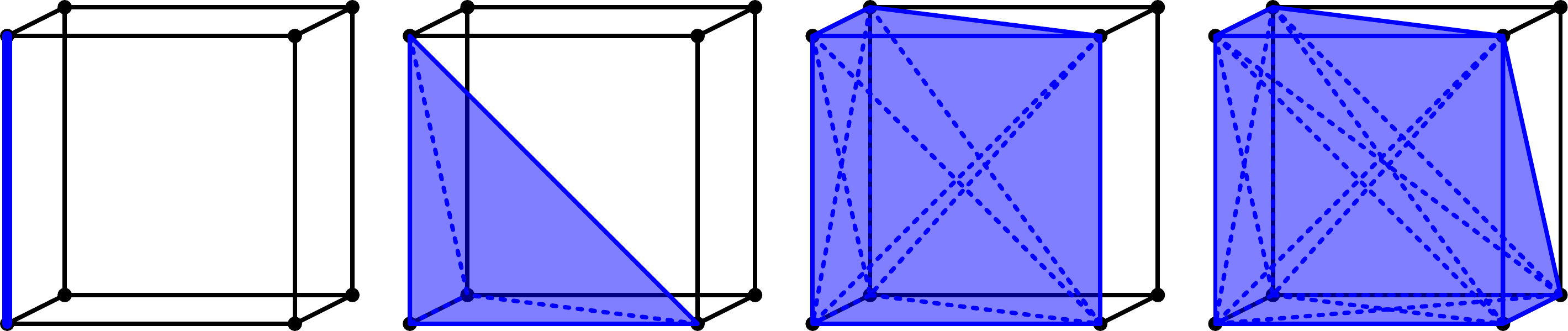}
\caption{A maximal simplex in $\N(\I_n,r)$ for $n=3$ and for $r=1,2,3,4$.
Complex $\N(\I_3,1)$ has twelve maximal edges; $\N(\I_3,2)$ has eight maximal tetrahedra (one for each vertex); $\N(\I_3,3)$ has twelve maximal $5$-simplices (one for each edge); and $\N(\I_3,4)$ has eight maximal $6$-simplices (one for each vertex).
The homotopy types are $\N(\I_3,1)\simeq \vee_5 \bs^1$; $\N(\I_3,2)\simeq \vee_7 \bs^2$; $\N(\I_3,3)\simeq \vee_3 \bs^4$; $\N(\I_3,4)\simeq \bs^6$.}
\label{fig:maximal}
\end{figure}

Table~\ref{table} shows the known homotopy types of $\N(\I_n,r)$.
The row $r=0$ follows since $\N(\I_n,0)$ is a disjoint union of vertices.
The row $r=1$ follows from an Euler characteristic computation, since $\N(\I_n,1)$ is a connected graph.
The diagonal $r=2n-2$ contains homeomorphisms $\N(\I_n,2n-2)\cong \bs^{2^n-2}$ with $(2^n-2)$-dimensional spheres, since $\N(\I_n,2n-2)$ is the boundary of the $(2^n-1)$-dimensional simplex with $2^n$ vertices.
One can check that $\N(\I_n, 2n-3)$ is equal to the complex $\Theta(\text{Cube}(n,1))$ from \cite{conant2010boolean}; see the table in their Theorem~5, and see their Example~3 for a proof that $\N(\I_3,3) = \Theta(\mbox{Cube}(3,1)) \simeq \vee_3 \bs^4$.
We prove the homotopy equivalences to wedge sums of 2-spheres in the row $r=2$ in \Cref{thm:main_for_index_2}.
All of the other entries are not yet known up to homotopy type.
When $(n,r)=(4,3), (5,3), (4,4), (5,4), (4,6)$, computer computations using Polymake~\cite{polymake:2000} give the following reduced homology groups:

\begin{align*}
\tilde{H}_i(\N(\I_4,3))&=\Z \text{ for }i=3 && \Z^{24} \text{ for }i=4 && 0 \text{ for }i\neq 3,4\\
\tilde{H}_i(\N(\I_5,3))&=\Z^9 \text{ for }i=3 && \Z^{120} \text{ for }i=4 && 0 \text{ for }i\neq 3,4\\
\tilde{H}_i(\N(\I_4,4))&=\Z \text{ for }i=4 && \Z^{10} \text{ for }i=6 && 0 \text{ for }i\neq 4,6\\
\tilde{H}_i(\N(\I_5,4))&=\Z^{11} \text{ for }i=4 && \Z^{60} \text{ for }i=6 && 0 \text{ for }i\neq 4,6\\
\tilde{H}_i(\N(\I_4,6))&=\Z^7 \text{ for }i=10 &&  && 0 \text{ for }i\neq 10\\
\end{align*}
In \Cref{thm:main3}, we show that for $n \geq 4$, the reduced homology group $\tilde{H}_i(\N(\I_n, 3);\Z)$ is nonzero if and only if $i \in \{3, 4\}$.

\begin{center}
\begin{table}[H]
\renewcommand{\arraystretch}{1.2}
\begin{tabular}{| >{$} c <{$} | >{$} c <{$} | >{$} c <{$} | >{$} c <{$} | >{$} c <{$} | >{$} c <{$} | >{$} c <{$} | >{$} c <{$} | >{$} c <{$} |} 
\hline
\N(\I_n,r) & n=1 & 2 & 3 & 4 & 5 & 6 & 7 & 8 \\
\hline
r=0 & \bs^0 & \vee_3 \bs^0 & \vee_7 \bs^0 & \vee_{15} \bs^0 & \vee_{31} \bs^0 & \vee_{63} \bs^0 & \vee_{127} \bs^0 & \vee_{255} \bs^0 \\
\hline
1 & * & \bs^1 & \vee_5 \bs^1 & \vee_{17} \bs^1 & \vee_{49} \bs^1 & \vee_{129} \bs^1 & \vee_{321} \bs^1 & \vee_{769} \bs^1 \\
\hline
2 & * & \bs^2 & \vee_7 \bs^2 & \vee_{31}\bs^2 & \vee_{111}\bs^2 & \vee_{351}\bs^2 & \vee_{1023}\bs^2 & \vee_{2815}\bs^2 \\
\hline
3 & * & * & \vee_3 \bs^4 & \beta_3=1; \beta_4=24
 
& \beta_3=9; \beta_4=120

& & 
& \\
\hline
4 & * & * & \bs^6 & \beta_4=1; \beta_6=10

& \beta_4=11; \beta_6=60

& 
& & \\
\hline
5 & * & * & * & \beta_{10}=7

& 
& & & \\
\hline
6 & * & * & * & \bs^{14} & & & & \\
\hline
7 & * & * & * & * & & & & \\
\hline
8 & * & * & * & * & \bs^{30} & & & \\
\hline
\end{tabular}
\caption{The known homotopy types or Betti numbers of $\N(\I_n,r)$.}
\label{table}
\end{table}
\end{center}

\section{The case of $r=2$}
\label{sec:r=2}

In this section, we characterize the homotopy type of $\N(\I_n,2)$ for all $n$.
It is easy to see that $\N(\I_1,2)$ is contractible.
For $n\ge 2$, the complex $\N(\I_n,2)$ has $2^n$ maximal $n$-simplices, each of the form $N_1[v]$ as $v$ varies over the vertices of $\I_n$.
We will show that these complexes are always homotopy equivalent to a wedge sum of 2-spheres.

We remark that in the seminal paper~\cite{lovasz}, Lov{\'a}sz considers simplicial complexes (called neighborhood complexes) which are generated by the open neighborhoods $N(v)$.
Using topology of the neighborhood complex of the graph $G$, Lov\'asz gave a general lower bound for the chromatic number of $G$.
Instead, in this section we are considering simplicial complexes $\N(G,2)$ generated by the closed neighborhoods $N_1[v]$, in the particular case when $G$ is a hypercube graph.
Let us now describe the homotopy types of $\N(\I_n,2)$.

\begin{theorem}
\label{thm:main_for_index_2}
Let $n\ge 2$.
Then the {\v C}ech complex $\N(\I_n,2)$ is homotopy equivalent to a wedge sum of $2^{n-2}(n^2-3n+4)-1$ spheres of dimension $2$, {\it i.e.}, 
\[\N(\I_n,2) \simeq \bigvee\limits_{2^{n-2}(n^2-3n+4)-1} \bs^2.\]
\end{theorem}

Note that the sequence $2^{n-2}(n^2-3n+4)-1$ is listed as the Bj\"{o}rner--Welker sequence $A055580(n-2)$ on the Online Encyclopedia of Integer Sequences (OEIS)~\cite{oeis,bjorner1995homology}.
The idea of the proof here is along the same lines as the proof of \cite[Theorem 1]{adamaszek2022vietoris}, as follows.
Recall that the metric space $\I_n$ is the set of all $2^n$ binary strings of length $n$, namely the numbers from $0$ to $2^n-1$ written in binary, equipped with the Hamming distance.
We furthermore consider the metric spaces $G_m$ for $m$ an integer, consisting of all numbers from $0$ to $m-1$ written as binary strings, also equipped with the Hamming distance.
To make our discussion simpler, we will also treat $G_m$ as the graph on $m$ vertices in which two vertices are adjacent if and only if their binary representations differ at exactly one place.
Note that $G_{2^n}=\I_n$, and so the metric spaces $G_m$ form a larger collection of metric spaces than just the hypercubes $\I_n$.
To prove \Cref{thm:main_for_index_2}, we determine the homotopy type of $\N(G_m,2)$ for all integers $m\ge 1$.

Let $m\geq 2$ be a non-negative integer with the following binary representation for $m-1$:
\[m-1=2^{i_1}+2^{i_2}+\cdots + 2^{i_k}, \text{ where } i_1 <i_2< \cdots <i_k.\]
Define $\alpha(m-1)= (k-1)^2$.

For $j\in [k]\coloneqq\{1,2,\ldots,k\}$, let $\lambda_j \coloneqq (m-1)^{i_j}= 2^{i_1}+\cdots+2^{i_{j-1}}+2^{i_{j+1}}+\cdots+2^{i_k}$, where the superscript notation $(m-1)^{i_j}$ is from \Cref{eq:superscript}.
Clearly, $N_{G_m}(m-1)=\{\lambda_1,\dots,\lambda_k\}$.
We now state a more general result, which will imply \Cref{thm:main_for_index_2}.

\begin{theorem}
\label{thm:general_for_index_2}
Let $m\ge 2$.
Then, the complex $\N(G_m,2)$ is homotopy equivalent to a wedge of $2$-dimensional spheres.
More precisely,
\begin{equation}
\label{eq:inductivestep}
\N(G_m,2) \simeq \N(G_{m-1},2) \vee \bigvee\limits_{\alpha(m-1)} \bs^2.
\end{equation}
\end{theorem}
\begin{proof}
The first part of \Cref{thm:general_for_index_2} follows from induction and \Cref{eq:inductivestep}.
Therefore it is enough to prove \Cref{eq:inductivestep}.
We will prove this using \Cref{lem:splitting} and by computing the link and deletion of vertex $m-1$ in $\N(G_m,2)$.

Our base case is that $\N(G_2,2)$ is contractible, {\it i.e.}, the $0$-fold wedge sum of $2$-spheres.
This is the same for $\N(G_3,2)$.
From the induction hypothesis, we know that $\N(G_{m-1},2)$ is homotopy equivalent to a wedge of spheres.
We first determine the homotopy type of the deletion complex, {\it i.e.}, $\N(G_m,2)\setminus \{m-1\}$.

\begin{claim}
\label{claim:1}
$\N(G_m,2)\setminus \{m-1\}\simeq \N(G_{m-1},2) \vee \bigvee\limits_{\binom{k-1}{2}}\bs^2$.
\end{claim}

\begin{proof}[Proof of \Cref{claim:1}]
It is easy to observe that
\[\N(G_m,2)\setminus \{m-1\} = \N(G_{m-1},2)\cup \Delta^{N_{G_m}(m-1)}\text{ and }\N(G_{m-1},2)\cap \Delta^{N_{G_m}(m-1)}= (\Delta^{N_{G_m}(m-1)})^{(1)},\]
where $\Delta^S$ denotes the simplex on vertex set $S$ and $(\Delta^S)^{(1)}$ denotes the $1$-dimensional skeleton of this simplex.
The $1$-skeleton of a $(k-1)$-dimensional simplex is homotopy equivalent to a wedge of $\binom{k-1}{2}$ circles, since it is a connected graph with Euler characteristic $k-\binom{k}{2}=\binom{k-1}{2}-1$.
Hence
\[(\Delta^{N_{G_m}(m-1)})^{(1)}\simeq \bigvee\limits_{\binom{k-1}{2}} \bs^1.\]
From the induction hypothesis, we know that $\N(G_{m-1},2)$ is a wedge of $2$-dimensional spheres and therefore the inclusion map $\iota \colon (\Delta^{N_{G_m}(m-1)})^{(1)} \hookrightarrow \N(G_{m-1},2)$ is null-homotopic.
Hence, using \Cref{lem:complex_union}, we get the following.
\begin{equation*}
\begin{split}
\N(G_m,2)\setminus \{m-1\} & \simeq \N(G_{m-1},2)\vee \Delta^{N_{G_m}(m-1)} \vee \susp\left(\bigvee\limits_{\binom{k-1}{2}} \bs^1\right)\\
& \simeq \N(G_{m-1},2) \vee \bigvee\limits_{\binom{k-1}{2}} \bs^2.
\end{split}
\end{equation*}
This completes the proof of \Cref{claim:1}.
\end{proof}

We now determine the homotopy type of the complex $\lk_{\N(G_m,2)}(m-1)$, so that we can use Lemma~\ref{lem:splitting}.
Observe that $\lk_{\N(G_m,2)}(m-1)$ is a simplicial complex whose maximal simplices are $N_{G_{m}}(m-1)$ and $N_{G_{m-1}}[\lambda_j]$ for $j\in [k]$, {\it i.e.}, 
\begin{equation}
\label{eq:cover of link}
\lk_{\N(G_m,2)}(m-1)= \Delta^{N_{G_{m}}(m-1)}\cup \Delta^{N_{G_{m-1}}[\lambda_1]}\cup \Delta^{N_{G_{m-1}}[\lambda_2]}\cup \cdots\cup \Delta^{N_{G_{m-1}}[\lambda_k]}.
\end{equation}  
Recall that the $\lambda_j$'s are neighbors of the vertex $m-1$ in the graph $G_m$, and that $\Delta^S$ denotes the simplex on the vertex set $S$.
For $1\leq r <s \leq k$, let $\lambda_{r,s}$ be the vertex of $G_{m-1}$ whose binary representation is $2^{i_1}+\cdots+2^{i_{r-1}}+2^{i_{r+1}}+\cdots+2^{i_{s-1}}+2^{i_{s+1}}+\cdots+2^{i_k}$.
Note that $\lambda_{r,s}=(m-1)^{i_r,i_s}$; see \Cref{eq:superscript}.
Clearly, $\lambda_{r,s}$ is adjacent to both $\lambda_r$ and $\lambda_s$ in $G_{m-1}$.
We now use \Cref{thm:nerve} and \Cref{eq:cover of link} to determine the homotopy type of $\lk_{\N(G_m,2)}(m-1)$.
The following are some easy observations from the definition of $G_m$.
\begin{enumerate}
\item Every member in the union of the right side of \Cref{eq:cover of link} is a simplex and hence contractible.
\item For any $1\le r <s \leq k$, $\Delta^{N_{G_{m-1}}[\lambda_r]}\cap \Delta^{N_{G_{m-1}}[\lambda_s]}=\Delta^{\{\lambda_{r,s}\}}$, which is a point and therefore contractible.
\item For any $1\leq i \leq k$, $\Delta^{N_{G_{m}}(m-1)}\cap \Delta^{N_{G_{m-1}}[\lambda_i]}=\Delta^{\{\lambda_i\}}$, which is again a point and therefore contractible.
\item The intersection of any three or more members from the union on the right side of \Cref{eq:cover of link} is always empty.
\end{enumerate}
  
The observations above along with \Cref{thm:nerve} imply that the complex $\lk_{\N(G_m,2)}(m-1)$ is homotopy equivalent to the nerve of $\{\Delta^{N_{G_{m}}(m-1)}, \Delta^{N_{G_{m-1}}[\lambda_1]}, \Delta^{N_{G_{m-1}}[\lambda_2]}, \dots, \Delta^{N_{G_{m-1}}[\lambda_k]}\}$, which is homotopy equivalent to the $1$-dimensional skeleton of a $k$-simplex on $k+1$ vertices.
Therefore,
\[\lk_{\N(G_m,2)}(m-1)\simeq (\Delta^k)^{(1)}\simeq \bigvee\limits_{\binom{k}{2}}\bs^1.\]
  
Thus, from induction and \Cref{claim:1}, we get that the inclusion map $\iota : \lk_{\N(G_m,2)}(m-1) \hookrightarrow \N(G_m,2)\setminus \{m-1\}$ is a null-homotopy (since the latter is homotopy equivalent to a wedge of $2$-dimensional spheres).
Hence, \Cref{lem:splitting} implies the following.
\begin{equation*}
\begin{split}
\N(G_m,2) & \simeq (\N(G_m,2)\setminus \{m-1\}) \vee \susp \lk_{\N(G_m,2)}(m-1)\\
& \simeq \N(G_{m-1},2)\vee \left(\bigvee\limits_{\binom{k-1}{2}}\bs^2\right) \vee \left(\bigvee\limits_{\binom{k}{2}}\bs^2\right)\\
& = \N(G_{m-1},2)\vee \bigvee\limits_{(k-1)^2}\bs^2.
\end{split}
\end{equation*}
This completes the proof of \Cref{thm:general_for_index_2}.
\end{proof}

The proof of \Cref{thm:main_for_index_2} now follows as a special case.

\begin{proof}[Proof of \Cref{thm:main_for_index_2}]
\Cref{thm:general_for_index_2} implies that $\N(\I_n,2) \simeq \bigvee\limits_{\beta(n)}\bs^2$, where
\begin{equation*}
\begin{split}
\beta(n)& = \sum\limits_{m=2}^{2^n} \alpha(m-1) = \sum\limits_{k=1}^{n} \binom{n}{k}(k-1)^2 = \sum\limits_{k=1}^{n} \binom{n}{k}k^2 - 2 \sum\limits_{k=1}^{n} \binom{n}{k} k + \sum\limits_{k=1}^{n} \binom{n}{k} \\
& = \sum\limits_{k=0}^{n} \binom{n}{k}k^2 - 2 \sum\limits_{k=0}^{n} \binom{n}{k} k + \sum\limits_{k=0}^{n} \binom{n}{k}-1 = 2^{n-2}(n^2+n)-2\cdot 2^{n-1}n+2^n-1\\ 
& = 2^{n-2}(n^2-3n+4)-1.
\end{split}
\end{equation*}
\end{proof}

\section{The case of $r=3$}
\label{sec:r=3}

In this section, we prove two main results about $\N(\I_n,r)$ at scale $r=3$.
First, we prove that $\N(\I_n,3)$ is homotopy equivalent to a simplicial complex of dimension at most 4 (\Cref{thm:collapsibility}).
Second, for $n\geq 4$, we show that the reduced homology $\tilde{H}_i(\N(\I_n, 3);\Z)$ is nonzero if and only if $i \in \{3, 4\}$ (\Cref{thm:main3}).
We remind the reader that the cases $n=1,2,3$ are understood, since $\N(\I_1,3)$ and $\N(\I_2,3)$ are contractible, and since $\N(\I_3,3) \simeq \vee_3 \bs^4$ by \cite{conant2010boolean}.

\begin{remark}
For any $n \geq 1$, let $\X_n = \N(\I_n, 3)$.
From \Cref{lem:maximal-simplices}, the maximal simplices of $\X_n$ will be of the type 
$N_1[v] \cup N_1[w]$ for some $v, w \in V(\I_n)$ that are adjacent, {\it i.e.}\ $ v \sim w$.
Since $v\sim w$, we remark that $v\in N(w)$ and $w\in N(v)$, and therefore $N_1[v] \cup N_1[w] = N(v) \cup N(w)$ when $v \sim w$.
\end{remark}

For each $i \in [n]$ and $\epsilon \in \{0, 1\}$, let $\I_n^{i, \epsilon}$ be the induced subgraph of $\I_n$ on the vertex set $\{v \in V(\I_n) : v(i) = \epsilon \}$.
Observe that we have an isomorphism $\I_n^{i, \epsilon} \cong \I_{n-1}$.

\subsection{Collapsibility}
\label{subsec:collapsibility}

We describe the concept of simplicial collapsibility, which will play an important role in the proof of \Cref{thm:main3}.
For more details, we refer the reader to \cite{Lew2018, Wegner1975}.
A standard notational convention is that if $S$ is a finite set, then we let $|S|$ denote the size of this set.

Let $X$ be a finite simplicial complex.
Let $\sigma \in X$ be simplex such that $|\sigma|\leq d$.
Suppose there is a unique maximal simplex $\gamma \in X$ that contains $\sigma$.  
Then an {\it elementary $d$-collapse} of $X$ is the subcomplex $Y$ of $X$ which  obtained by
removing all those simplices $\delta$ of $X$ satisfying
$\sigma \subseteq \delta \subseteq \gamma$.
We denote this elementary $d$-collapse by  $X \xrightarrow{\sigma} Y$.

We remark that if $\gamma$ is not equal to $\sigma$, then the elementary $d$-collapse $X \xrightarrow{\sigma} Y$ is a simplicial collapse, and therefore does not change the homotopy type of the space.
However, a $d$-collapse allows the possibility that $\gamma=\sigma$; in this case performing the $d$-collapse changes the homotopy type.

The complex $X$ is called \emph{$d$-collapsible} if for some integer $k$, there exists a sequence of elementary $d$-collapses 
\[
X=X_0\xrightarrow{\sigma_1} X_1 \xrightarrow{\sigma_2} \cdots \xrightarrow{\sigma_{k}} X_k=\emptyset
\]
from $X$ to the void complex $\emptyset$.
It is easy to check that if $X$ is $d$-collapsible and $d < c$, then $X$ is $c$-collapsible. 

\begin{definition}
For a simplicial complex $X$, the \emph{collapsibility number} of $X$ is the minimal integer $d$ such that $X $ is $d$-collapsible.
\end{definition}

\begin{example}\label{example:collapsibility}
Let $X$ be the simplicial complex on the vertex set $\{1, 2, 3, 4\}$ whose set of maximal simplices is $\{\{1, 2, 4\}, \{2, 3\}, \{3, 4\}\}$; see \Cref{Collapsibility_Figure}(A).
Then we have the following sequence of elementary $2$-collapses from $X$ to $\emptyset$ (see \Cref{Collapsibility_Figure}):
\[
X\xrightarrow{\{1\}} X_1 \xrightarrow{\{2, 3\}} X_2 \xrightarrow{\{3\}} X_3 \xrightarrow{\{2\}} X_4 \xrightarrow{\{4\}}  \emptyset.
\]
Hence $X$ is $2$-collapsible.
One can check that $X$ is not $1$-collapsible.
Thus the collapsibility number of $X$ is $2$.
\end{example}
	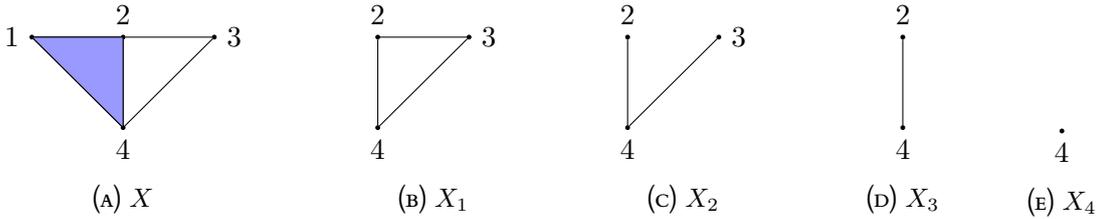
\begin{figure}[H]
		\begin{subfigure}[]{0.3\textwidth}
		\centering
		\begin{tikzpicture}
			[scale=0.4, vertices/.style={draw, fill=black, circle, inner sep=0.5pt}]
			\node[vertices, label=left:{$1$}] (a) at (1, 1) {};
			\node[vertices, label=above:{$2$}] (b) at (4, 1) {};
			\node[vertices, label=right:{$3$}] (c) at (7, 1) {};
			\node[vertices, label=below:{$4$}] (d) at (4, -2) {};
			\foreach \to/\from in {a/b, b/c, b/d, c/d, a/d}
		\draw [-] (\to)--(\from);
			\filldraw[fill=blue!40, draw=black] (1, 1)--(4, 1)--(4, -2)--cycle;
			
		\end{tikzpicture}
		\caption{$X$} \label{a}
		\end{subfigure}
		\begin{subfigure}[]{0.2\textwidth}
				\centering
			\begin{tikzpicture}
				[scale=0.4, vertices/.style={draw, fill=black, circle, inner sep=0.5pt}]
				\node[vertices, label=above:{$2$}] (b) at (4, 1) {};
				\node[vertices, label=right:{$3$}] (c) at (7, 1) {};
				\node[vertices, label=below:{$4$}] (d) at (4, -2) {};
				\foreach \to/\from in { b/c, b/d, c/d}
				\draw [-] (\to)--(\from);
				
			\end{tikzpicture}
			\caption{$X_1$} \label{b}
			\end{subfigure}
				\begin{subfigure}[]{0.2\textwidth}
				\centering
				\begin{tikzpicture}
					[scale=0.4, vertices/.style={draw, fill=black, circle, inner sep=0.5pt}]
					\node[vertices, label=above:{$2$}] (b) at (4, 1) {};
					\node[vertices, label=right:{$3$}] (c) at (7, 1) {};
					\node[vertices, label=below:{$4$}] (d) at (4, -2) {};
					\foreach \to/\from in {b/d, c/d}
					\draw [-] (\to)--(\from);
					
				\end{tikzpicture}
				\caption{$X_2$} \label{c}
			\end{subfigure}
				\begin{subfigure}[]{0.15\textwidth}
				\centering
				\begin{tikzpicture}
					[scale=0.4, vertices/.style={draw, fill=black, circle, inner sep=0.5pt}]
					\node[vertices, label=above:{$2$}] (b) at (4, 1) {};
					\node[vertices, label=below:{$4$}] (d) at (4, -2) {};
					\foreach \to/\from in {b/d}
					\draw [-] (\to)--(\from);
					
				\end{tikzpicture}
				\caption{$X_3$} \label{d}
			\end{subfigure}
   	\begin{subfigure}[]{0.1\textwidth}
    \vspace{ 1.8 cm}
				\centering
				\begin{tikzpicture}
					[scale=0.4, vertices/.style={draw, fill=black, circle, inner sep=0.5pt}]
					\node[vertices, label=below:{$4$}] (d) at (4, -2) {};
					
				\end{tikzpicture}
				\caption{$X_4$} \label{d}
			\end{subfigure}
   \caption{The simplicial complex $X$ is $2$-collapsible.}\label{Collapsibility_Figure}
	\end{figure}
We have the following consequence of $d$-collapsibility of a simplicial complex.

\begin{proposition}[\cite{Wegner1975}]
\label{prop:collapsibilitysubcomplex}
If $X$ is $d$-collapsible then it is homotopy equivalent to a simplicial complex of dimension smaller than $d$.
\end{proposition}

We now define a minimal exclusion sequence. Let $X$ be a simplicial complex on the vertex set $[n]$.
Suppose we fix a linear order $\prec$ on the set of maximal simplices of $X$.
Let $ \gamma_1 \prec \ldots \prec \gamma_m$ be the maximal simplices of $X$, ordered as such.
Let $\gamma \in X$.
The \textit{minimal exclusion sequence} $\mes_{\prec}(\gamma)$ is a sequence of vertices from $\gamma$, defined as follows.
Let $j$ be the smallest index such that $\gamma \subseteq \gamma_j$.
If $j=1$, then $\mes_{\prec}(\gamma)$ is the null sequence.
If $j\ge 2$, then $\mes_{\prec}(\gamma)=(a_1,\ldots, a_{j-1})$ is the finite sequence of length $j-1$ such that
$a_1=\min (\gamma \setminus \gamma_1) \in  [n]$ and  for each $k\in\{2, \ldots, j-1\}$, 
\[a_k=\begin{cases}
\min(\{a_1,\dots,a_{k-1}\}\cap (\gamma \setminus \gamma_k)) & \text{if } \{a_1,\dots,a_{k-1}\}\cap (\gamma \setminus \gamma_k)\neq\emptyset,\\
\min (\gamma\setminus \gamma_k) & \text{otherwise.}
\end{cases} \]

Let $M_{\prec}(\gamma)\subseteq \gamma$ denote the set of vertices that appear in $\mes_{\prec}(\gamma)$.
Define
\[d_{\prec}(X)\coloneqq\max_{\gamma \in X}|M_{\prec}(\gamma)|.\]

\begin{proposition}[Theorem~6 of~\cite{Lew2018}]
\label{prop:minimalexclusion}
If $\prec$ is a linear ordering of the maximal simplices of $X$, then $X$ is $d_{\prec}(X)$-collapsible.
\end{proposition}

\begin{example}
Let $X$ be the simplicial complex of \Cref{example:collapsibility}.
Let $\{1, 2,4\} \prec \{2, 3\} \prec \{3, 4\}$ be a linear order on the maximal simplices.
Since there are only $3$ maximal simplices,  $\mes_{\prec}(\gamma)$ is a sequence of length at most $2$ for any simplex $\gamma \in X$.
Thus $d_{\prec}(X) \leq 2$.
One can check that $\mes_{\prec}(\{3, 4\}) = (3, 4)$, and so $M_{\prec}(\{3,4\})=\{3,4\}$.
Hence $d_{\prec}(X) = 2$ and therefore \Cref{prop:minimalexclusion} implies that  $X$ is $2$-collapsible.
\end{example}

We will use \Cref{prop:minimalexclusion} to upper bound the collapsibility number of $\X_n = \N(\I_n, 3)$, and we will use the retracts in \Cref{lem:retraction} to lower bound the collapsibility number.

\begin{lemma}
\label{lem:retraction}
Let $n \ge m$ and let $Q$ be an $m$-dimensional cube subgraph of $\I_n$.
Then there exists a retraction of $\X_n$ onto  $\N(Q, 3)$.
\end{lemma}

\begin{proof}
The case $n=m$ is clear.
And, it is enough to prove the result for $m = n-1$, since then the general case of $n>m$ can be obtained by composition.
For $m=n-1$, there exist  $i \in [n]$ and $\epsilon \in \{0, 1\}$ such that $Q = \I_n^{i, \epsilon}$.
Define $\phi: V(\I_n) \to  V(Q)$ as follows: for $v \in V(\I_n)$ and $t \in [n]$, let
\[
\phi(v)(t)  = \begin{cases}
\ v(t)& \text{if} \ t \neq i,\\
\ \epsilon & \text{if} \ t =  i.\\
\end{cases}
\]
	
We  extend the map $\phi$ to $\tilde{\phi} : \X_n \to \N(Q, 3)$ by 
$\tilde{\phi}(\sigma) \coloneqq \{\phi(v) : v \in \sigma\}$ for all $\sigma \in \X_n$.
Clearly, $\tilde{\phi}$ is surjective.
Observe that for any $v \sim w$, if $w = v^{i}$ (recall~\eqref{eq:superscript}) then $\phi(v) = \phi(w)$, and $w \neq v^{i}$ implies that $\phi(v) \sim \phi(w)$.
Hence, for $\sigma \subseteq N(u_1) \cup N(u_2)$, where $u_1 \sim u_2$, we get that  $\tilde{\phi}(\sigma) \subseteq N(\phi(u_1)) \cup N(\phi(u_2))$.
Thus, $\tilde{\phi}$ is a simplicial map that does not move the simplices of $\N(Q, 3)$, and hence $\tilde{\phi}$ is a retraction.
\end{proof}

\begin{lemma}
\label{lem:4neighbor}	
Let $\sigma \in \X_n$ be a maximal simplex.
If $|N(v) \cap \sigma| \geq 3$ for some $v$, then $N(v) \subseteq \sigma$.
\end{lemma}

\begin{proof}
Without loss of generality, using the symmetries of the hypercube, assume that $\{v^1, v^2, v^3\} \subseteq \sigma$.
Let $\sigma = N(u) \cup N(w)$, where $u \sim w$.

First, suppose $\{u, w\} \cap   \{v^1, v^2, v^3\} = \emptyset$.
Since $v^1 \in \sigma$, either $v^1 \in N(u)$ or $v^1 \in N(w)$;
without loss of generality assume that $v^1 \in N(u)$.
So $u= v^{1, j_0}$ for some $j_0 \neq 1$.
If $j_0 = 2$, then $v^3 \not\sim u$ and therefore $v^3 \sim w$.
Then $w = v^{3, k_0}$ for some $k_0 \in [n]$.
But then $d(u, w) \geq 2$, a contradiction.
If $j_0 \neq 2$, then $v^2 \not\sim u$ and so $v^2 \sim w$.
But then $d(u, w) \geq 2$ again, a contradiction.

Hence $\{u, w\} \cap   \{v^1, v^2, v^3\} \neq \emptyset$.
Without loss of generality assume that $u \in \{v^1, v^2, v^3\}$ and $u = v^1$.
If $w \neq v$, then there exists $i_0 \neq 1$ such that $w = v^{1, i_0}$.
Clearly $\{v^2, v^3\} \not\subseteq N(v^{1, i_0})=N(w)$.
Since $\{v^2, v^3\}\not\subseteq N(u)$, we see that $\{v^1, v^2, v^3\} \not\subseteq \sigma$, a contradiction.
Thus we conclude that $w = v$, and hence $N(v) \subseteq \sigma$.
\end{proof}

\begin{theorem}
\label{thm:collapsibility}
For $n \geq 4$, the collapsibility number of $\X_n$ is $5$.
\end{theorem}

\begin{proof}
The proof is similar to the proof of~\cite[Theorem $1.5$]{shukla2022vietoris}.
We first show that the collapsibility number of $\X_n$ is at most $5$.
To begin, we fix some notation.
Choose a linear order $\prec$ on the set of all maximal simplices of $\X_n$, namely $\gamma_1 \prec  \gamma_2  \prec \ldots \prec \gamma_{q}$.
Let $\sigma \in \X_n$ be arbitrary, and let $p$ be the smallest index such that  $\sigma \subseteq \gamma_p$.
Let  $\mes_{\prec}(\sigma) = (\alpha_1, \ldots, \alpha_{p-1})$.
There exist $v, w \in V(\I_n)$ such that $v \sim w$ and   $\gamma_p = N(v) \cup N(w)$.
We want to show that $|M_{\prec}(\sigma)| \leq 5$

We first show that $|M_{\prec} (\sigma) \cap N(v)| \leq 3$.
On the contrary assume that  $|M_{\prec}(\tau) \cap N(v)| \geq 4$.
Let $l$ be the least integer  such that $|\{\alpha_1, \ldots, \alpha_l\} \cap N(v)| = 3$.
Let  $\{\alpha_1, \ldots, \alpha_l\} \cap N(v) = \{\alpha_{i_1}, \alpha_{i_2}, \alpha_{i_3} \} $.
Clearly, $l < p-1$ and  $\alpha_l \in \{\alpha_{i_1}, \alpha_{i_2}, \alpha_{i_3}\}$.

If $ \{\alpha_1, \ldots, \alpha_{l}\} \cap (\sigma \setminus \gamma_{l+1}) \neq  \emptyset$, then $\alpha_{l+1} \in \{\alpha_1, \ldots, \alpha_l\}$, and therefore $\{\alpha_1, \ldots, \alpha_{l+1}\} \cap N(v) = \{\alpha_{i_1}, \alpha_{i_2}, \alpha_{i_3}\}$.
On the other hand if $ \{\alpha_1, \ldots, \alpha_{l}\} \cap (\sigma \setminus \gamma_{l+1}) = \emptyset$, then $\{\alpha_{i_1}, \alpha_{i_2}, \alpha_{i_3}\} \subseteq \gamma_{l+1}$.
Then   $N(v)\subseteq \gamma_{l+1}$ by  \Cref{lem:4neighbor}.
Thus by definition of $\alpha_{l+1}$, we see that  $\alpha_{l+1} \notin N(v)$. Hence 	$\{\alpha_1, \ldots, \alpha_{l+1}\} \cap N(v) = \{\alpha_{i_1}, \alpha_{i_2}, \alpha_{i_3}\}$.
If $l+1 = p-1$, then we  get that $|M_{\prec} (\sigma) \cap N(v) | \leq 3$, which is  a contradiction as $|M_{\prec}(\sigma) \cap N(v)| \geq 4$.
Inductively assume that for  all $l \leq k < p-1$, we have $\{\alpha_1, \ldots, \alpha_{k}\} \cap N(v) = \{\alpha_{i_1}, \alpha_{i_2}, \alpha_{i_3}\}$.
If $\{\alpha_1, \ldots, \alpha_{p-2}\} \cap (\sigma \setminus \gamma_{p-1}) \neq  \emptyset $, then $\alpha_{p-1} \in \{\alpha_1, \ldots, \alpha_{p-2}\}$.
Hence $\{\alpha_1, \ldots, \alpha_{p-1}\} \cap N(v) = \{\alpha_{i_1}, \alpha_{i_2}, \alpha_{i_3}\}$.
If $\{\alpha_1, \ldots, \alpha_{p-2}\} \cap (\sigma \setminus \gamma_{p-1}) = \emptyset$, then $\{\alpha_{i_1}, \alpha_{i_2}, \alpha_{i_3}\} \subseteq \gamma_{p-1}$.
We get that  $N(v)\subseteq \gamma_{p-1}$ by \Cref{lem:4neighbor}.
Thus $\alpha_{p-1} \notin N(v)$.
It follows that $\{\alpha_1, \ldots, \alpha_{p-1}\} \cap N(v) = \{\alpha_{i_1}, \alpha_{i_2}, \alpha_{i_3}\}$, a contradiction as $|M_{\prec}(\sigma) \cap N(v)| \geq 4$.
Therefore  $|M_{\prec} (\sigma) \cap N(v) | \leq 3$.

By an argument similar to the above, we also have that $|M_{\prec} (\sigma) \cap N(w) | \leq 3$.
Since $M_{\prec}(\sigma)\subseteq \sigma \subseteq \gamma_p = N(v)\cup N(w)$, we see that  $|M_{\prec}(\sigma)| \leq 6$.
Suppose $|M_{\prec}(\sigma)| = 6$.
Let $M_{\prec} (\sigma)= \{z_1,z_2, z_3,  z_{4}, z_5, z_6\}$, where $M_{\prec}(\sigma) \cap N(v) = \{z_1, z_2, z_3\}$
and $M_{\prec}(\sigma)  \cap N(w) = \{z_4, z_5, z_6\}$.
Since $\X_n = \N(\I_n, 3)$ is defined using scale parameter $3$, we also have $N(v) = \{z_1, z_2, z_3\}$
and $N(w) = \{z_4, z_5, z_6\}$.
Without loss of generality we can assume that $z_6$ is appearing after $z_1, \ldots, z_5 $ in $\mes_{\prec}(\sigma)$.
Let $i_0$ be the first index such that $\alpha_{i_0} = z_6$; from the definition of $\mes(\sigma)$ we have $i_0 < p$.
Then $(\sigma \setminus \gamma_{i_0}) \cap \{z_1, \ldots, z_5\} = \emptyset$, which implies that $N(v) \cup \{z_4, z_5\} \subseteq \gamma_{i_0}$.
There exists some $j_0$ such that $\gamma_{i_0} = N(v) \cup N(v^{j_0})$.
Since $w \sim v$, we have $w = v^{k_0}$ for some $k_0 \in [n]$.
Since $z_4, z_5 \in \sigma \subseteq \gamma_p$, we see that  $z_4, z_5 \sim v^{k_0}$.
There exist $l$ and $s$ such that $z_4 = v^{k_0, l}$ and $z_5 = v^{k_0, s}$.
Since $z_4$ and $z_5$ are not adjacent to $v$, we conclude that $z_4, z_5 \in N(v^{j_0})$.
But this is possible only when $j_0 = k_0$, which implies that $\sigma_p = \sigma_{i_0}$, contradicting the fact that $i_0 < p$.
Hence $|M_{\prec}(\sigma)| \leq 5$.

Since $\sigma$ is an arbitrary simplex of $\X_n$, we have that $d_{\prec}(\X_n)=\max_{\sigma \in X}|M_{\prec}(\sigma)| \leq 5$.
It therefore follows from \Cref{prop:minimalexclusion} that $\X_n$ is $5$-collapsible.
	
Let $K$ be the \v{C}ech complex of a $4$-dimensional cube subgraph  of $\I_n$ at scale $r=3$, {\it i.e.}, $K\cong \X_4$.
Then using \Cref{lem:retraction}, there exists a retraction  $r:\X_n \to K$ for $n\ge 4$.
Since $K \cong \X_4$ and since we know $\tilde{H}_4(\X_4;\Z) \neq 0$ from a homology computation (\Cref{table}), we see that $\tilde{H}_4(K;\Z) \neq 0$.
Further, since the homomorphism $r_{\ast}: \tilde{H}_4(\X_n;\Z) \to \tilde{H}_4(K;\Z)$ induced by the retraction $r$ is surjective, $\tilde{H}_4(\X_n;\Z) \neq 0$.
Using  \Cref{prop:collapsibilitysubcomplex}, we conclude that the collapsibility number of $\X_n$ is $5$.
\end{proof}

\subsection{Homology}
\label{subsec:homology}

In this subsection we will prove \Cref{thm:main3}, which states that for $n\geq 4$, $\tilde{H}_i(\N(\I_n, 3);\Z) \neq 0$ if and only if $i \in \{3, 4\}$.
The proof will proceed as follows.
\Cref{thm:collapsibility} and \Cref{prop:collapsibilitysubcomplex} imply that 
$\tilde{H}_i(\N(\I_n, 3);\Z) = 0$ for $i \geq 5$.
\Cref{table} displays a homology computation on a computer that $\tilde{H}_i(\N(\I_4, 3);\Z) \neq 0$ for $i \in \{3,4\}$, and so using the retractions in \Cref{lem:retraction}, we get that $\tilde{H}_i(\N(\I_n, 3);\Z) \neq 0$ for $i \in \{3, 4\}$ and $n\ge 4$.
It therefore suffices to show that $\tilde{H}_i(\N(\I_n, 3);\Z) = 0$ for $0 \leq i \leq 2$, and we now build up the lemmas and the machinery in order to do this.

Recall that for each $i \in [n]$ and $\epsilon \in \{0, 1\}$, $\I_n^{i, \epsilon}$ is the induced subgraph of $\I_n$ on the vertex set $\{v \in V(\I_n) : v(i) = \epsilon \}$.
For $1 \leq i \leq n$ and $ \epsilon  \in  \{0, 1\}$, let $\X_n^{i,\epsilon} = \N(\I_n^{i, \epsilon}, 3)$ and let 
\[\partial(\X_n) = \bigcup\limits_{i \in [n],  \epsilon \in \{0, 1\}} \X_n^{i,\epsilon}.\]
We say that a simplex $\sigma \in \X_n$ \emph{covers all places} if for each $i \in [n]$, there exist vertices $v, w \in \sigma$ such that $v(i)= 1$ and $w(i) = 0$.

The following lemma plays a key role in the proof of \Cref{thm:main3}, when showing that $\X_n$ is 2-connected.
The lemma is applied there with $l=2$, but we think it is also of independent interest for all $l\le n-2$.

\begin{lemma}
\label{lem:homology} 
Let $n \geq 4$ and let $l \leq n-2$.
Then   any $l$-cycle $z$ in $\X_n$ is homologous to an $l$-cycle $\tilde{z}$ in $\partial(\X_n)$.
\end{lemma}

\begin{proof}
The proof is similar to the proof of~\cite[Lemma $4.15$]{shukla2022vietoris}.
Recall that for $v = v_1 \ldots v_n \in V(\I_n)$, we have the notation $v(i) = v_i$ for $1 \leq i \leq n$.

For any $k \in [n]$ and $\sigma \in \X_n$, we say that $\sigma$ \emph{covers $k$ places} if there exist distinct indices $i_1, \ldots, i_k \in [n]$ such that for each $1 \leq t \leq k$, we have vertices $a,b \in \sigma $ with $a(i_t) = 0$ and $b(i_t) = 1$.
For a chain $c = \sum c_i \tau_i$ in $\X_n$, if $c_i \neq 0$, then  we say that $\tau_i \in c$.
For a cycle $c$ in $ \X_n$, let $I(c) = \{\sigma \in c: \sigma \notin \partial(\X_n)\}$.

Let $z$ be an $l$-cycle in $\X_n$.
Clearly, if $I(z) = \emptyset$, then $z$ is an $l$-cycle in $\partial(\X_n)$.
So assume that  $I(z) \neq \emptyset$. We show that $z$ is homologous to a cycle $\tilde{z}$ such that $I(\tilde{z}) = \emptyset$. It is sufficient to show that 
$z$ is homologous to an $l$-cycle $z_1$ such that $|I(z_1)| < |I(z)|$, since then we can proceed by induction.
Let $\tau \in z$ be such that $\tau \in I(c)$.
Then $\tau$ covers all places.
Let $\gamma$ be a maximal simplex such that $\tau \subseteq \gamma$.
There exists $v \sim w$ such that 
$\gamma = N(v) \cup N(w)$.
		
If $N(w) \cap \tau  = \emptyset$, then $\tau \subseteq N(v)$.
Since $\tau$ is $l$-dimensional, we see that $\tau$ can cover at most $l+1$ places, which is a contradiction to our  assumption that $\tau$ covers all places as  $n > l+1$.
Hence $N(w)  \cap \tau \neq \emptyset$.
Similarly, $N(v) \cap \tau \neq \emptyset$.
Hence $N(v) \cap \tau \neq \emptyset$ and $N(w) \cap \tau \neq \emptyset$.

Since $w \sim v$, there exists $p \in [n]$ such that $w = v^{p}$.
First assume that  $v, w \in \tau$.
If $N(w) \cap \tau = \{v\}$, then $\tau = \{v, w,  v^{i_1}, \ldots, v^{i_{l-1}} \}$ for distinct $i_1, i_2, \ldots, i_{l-1} \in [n] \setminus \{p\}$.
In this case  $\tau$ covers only $l < n$ places, namely $i_1, \ldots, i_{l-1}, p$, a contradiction to the assumption that $\tau$ covers all places.
Hence $|N(w) \cap \tau | \geq 2$.
Then 
$\tau = \{v, w, v^{i_1}, \ldots, v^{i_s}, v^{p, j_1}, \ldots, v^{p, j_t}\}$, for some  $i_1, \ldots, i_s, j_1 \ldots, j_t \in [n]$, where $s+t = l-1$.
Here $\tau$ can cover at most $l$ places, namely $i_1, \ldots, i_s, j_1, \ldots, j_t, p$ (and furthermore $\tau$ covers $l$ places only if  
$\{i_1, \ldots, i_s\} \cap \{j_1, \ldots, j_t\} = \emptyset$).
Since $l < n$, $\tau$ does not cover all places.
Hence $v$ and $w$ cannot both be in $\tau$, {\it i.e.}\ $\{v, w\} \not\subseteq \tau$.

Suppose $v\in \tau$.
Then $w \notin \tau$.
If $N(w) \cap \tau = \{v\}$, then $\tau = \{v, v^{i_1}, \ldots, v^{i_l} \}$ for some $i_1, i_2, \ldots, i_l \in [n]$.
Observe that $\tau$ covers only $l$ places, namely $i_1, \ldots, i_l$.
Hence $|N(w) \cap \tau | \geq 2$.
Let $\tau = \{v, v^{i_1}, \ldots, v^{i_s}, v^{p, j_1}, \ldots, v^{p, j_t}\}$, where $w = v^{p}$, where $i_1, \ldots, i_s, j_1 \ldots, j_t \in [n]$, and where $s+t = l$.
Here $\tau$ can cover at most $l+1<n$ places, namely  $i_1, \ldots, i_s, j_1, \ldots, j_t, p$ (and furthermore $\tau$ covers $l+1$ places only if  
$\{i_1, \ldots, i_s\} \cap \{j_1, \ldots, j_t\} = \emptyset, p \notin \{i_1, \ldots, i_s\}$).
Thus, we conclude that  $v \notin \tau$.
By a similar argument, $w \notin \tau$.

For a simplex $\delta$, let $Bd(\delta)$ denote the simplicial boundary of $\delta$.
Let $(-1)^i z_{\tau}$ be the   coefficient of $\tau$ in $z$, where $z_\tau\in\{1,2,3,\ldots\}$.
Let $(-1)^j$ be the coefficient of $\tau$ in $Bd(\tau \cup \{v\})$.
Define an $l$-cycle  $z_1$ as follows:
\[
z_1  = \begin{cases}
\ z - z_{\tau} Bd(\tau \cup \{v\})& \text{if} \  i \ \text{and} \  j \ \text{are of same parity, and}\\
\ z+ z_{\tau}  Bd(\tau \cup \{v\}) & \text{if} \  i \ \text{and} \ j \ \text{are of opposite parity}.
\end{cases}
\]

Clearly, $z$ is homologous to $z_1$.
Let $\sigma \subseteq \tau \cup \{v\}$ be such that $|\sigma| = l+1$ and $\sigma \neq \tau$.
Then $v \in \sigma$, and so from the argument in the paragraphs above we conclude that  $\sigma \in \partial(\X_n)$.
Hence $|I(z_1)| < |I(z)|$.
	
\end{proof}

We first show that $\X_n$ is simply connected, and then we show it is 2-connected.

\begin{lemma}\label{lem:simplyconnected}
For $n \geq 2$, $\X_n$ is simply connected.
\end{lemma}

\begin{proof}
For any two vertices $v, w \in V(\I_n)$, we recall the distance $d(v, w)$ between $v$ and $w$ is the number of coordinates where $v$ and $w$ differ,
{\it i.e.,} $d(v, w) = |\{i: v(i) \neq w(i)\}|$. 

Let $\gamma$ be a closed path in $\X_n$.
Since $\X_n$ is a simplicial complex, $\gamma$ is homotopic to a closed path $\delta= v_1 v_2\ldots v_m v_1$, where $\{v_i, v_{i+1}\}$ and $\{v_m, v_1\}$ are edges in $\X_n$ for each $1 \leq i \leq m-1$.
Without loss of generality, we can assume that successive
vertices in $\delta$ are distinct.
Otherwise, if $v_i = v_{i+1}$ for some $i$, we could delete $v_i$ and
obtain a homotopic path.
If for some $i$, $d(v_i, v_{i+1}) =3$, then there exist $u_i, w_i $ such  that $d(v_i, w_i) = d(w_i, u_i) = d(u_i, v_{i+1}) = 1$ and $\{v_i, w_i,u_i,  v_{i+1}\} \in \X_n$.
Clearly, the path $\delta_1 = v_1 \ldots v_i w_i u_i v_{i+1} \ldots v_m v_1$ is homotopic to $\delta$.
Similarly, if for some $i$, $d(v_i, v_{i+1}) =2$, then there exists $a_i $ such  that $d(v_i, a_i) = d(a_i, v_{i+1}) = 1$ and $\{v_i, a_i, v_{i+1}\} \in \X_n$.
Clearly, the path $\delta_2 = v_1 \ldots v_i a_i v_{i+1} \ldots v_m v_1$ is homotopic to $\delta$.
Hence, by inserting new vertices between each such pair of the vertices of distance $2$ and  $3$, we can assume that $d(v_m, v_1)= 1 =d(v_i, v_{i+1}) $ for all $1 \leq i \leq m-1$.
If for some $i$, $v_i = v_{i+2 (\text{mod} \ m)}$, then we could delete $v_{i+1}$ and obtain a homotopic path.
Hence we can also assume that $v_i \neq v_{i+2 (\text{mod} \ m)}$ for all $i$.
  
Our proof is by induction on $n$.
If $n = 2$, then  clearly $\X_2$ is contractible and therefore any path is homotopic to a constant path.
Now let $n \geq 3$ and assume that $\X_r$ is simply connected for all $2 \leq r < n$.
We will show that the closed path $\delta$ in $\X_n$ is homotopic to a closed path which lies in $\X_{n-1}$.
Let $l$ be the least integer such that $v_l(n) \neq v_1(n)$.
Clearly, $l\geq 2$.
Since $d(v_{l-1}, v_l) = 1$, we have $v_{l-1}= v_{l}^{n}$.
Let us first assume that 
$v_{l+1 (\text{mod} \ n)} \neq v_1$, {\it i.e.,} $l \neq n$.
Since $v_{l+1} \neq v_{l-1}$, $d(v_{l-1}, v_{l+1}) = 2$ and therefore there exists a vertex $w$ such that $\{v_{l-1}, v_{l}, v_{l+1}, w\}$ are vertices of a square in $\I_n$.
Here, $d(v_{l-1}, w) = 1$  and $d(v_l, w) = 2$.
Since $\{v_{l-1}, v_{l}, v_{l+1}, w\}$ are vertices of a square, $w = v_l^{n, i_o}$ for some $i_o \neq n$.
Observe that  $w(n) = v_{l-1}(n) = v_1(n)$.
Clearly, the path $\delta_1 = v_1 \ldots v_{l-1} w v_{l+1} \ldots v_n v_1$ is homotopic to $\delta$.
By repeating this process, after a finite number of steps we get a path $\delta_k =  u_1 \ldots u_q$ which is  homotopic to $\delta$ and 
which satisfies $u_i(n) = u_1(n)$ for all $1 \leq i \leq q$.
Hence $\delta_k$ is a path in $\N(\I_n^{n, v_1(n)}) \simeq \X_{n-1}$.
By the induction hypothesis, $\delta_k$ is homotopic to a constant path, and therefore $\delta$ is also homotopic to a constant path.
\end{proof}

\begin{lemma}\label{lem:upto3}
For any $n \geq 2$, $\X_n$ is $2$-connected.
\end{lemma}

\begin{proof}
Since $\X_n$ is simply connected by \Cref{lem:simplyconnected}, by the Hurewicz theorem (\cite[Theorem~4.32]{Hatcher}) it is enough to show that  $\tilde{H}_2(\X_n)  = 0$.
Since $\X_3 \simeq \vee_3 \bs^4$ by~\cite[Example~3]{conant2010boolean}, clearly $\tilde{H}_2(\X_3)= 0$.
So fix $n \geq 4$, and inductively assume that $\tilde{H}_2(\X_m) = 0$ for all $2 \leq m < n$.
From \Cref{lem:homology}, any $2$-cycle in $\X_n$ is homologous to a $2$-cycle in $\partial(\X_n)$.
Hence it is sufficient to show that $\tilde{H}_2(\partial(\X_n)) = 0$.
For each $i \in [n]$ and $\epsilon \in \{0, 1\}$, recall 
$\X_n^{i, \epsilon} = \N(\I_n^{i, \epsilon}, 3)$.
Note  
$\partial(\X_n) = \bigcup\limits_{i \in [n], \epsilon \in \{0, 1\}} \X_n^{i, \epsilon}$.
Observe that each non-empty intersection $\X_n^{i_1, \epsilon_1} \cap \ldots \cap \X_n^{i_t, \epsilon_t}$ is homeomorphic to the \v{C}ech complex of some cube subgraph of dimension less than $n$ and therefore it is $2$-connected by the induction hypothesis.
Hence by \Cref{thm:nerve} $(ii)$,  $\partial(\X_n)$ is $2$-connected if and only if  $\mathbf{N}(\{\X_n^{i, \epsilon}\})$ is $2$-connected.
We now show that $\mathbf{N}(\{\X_n^{i, \epsilon}\})$ is $2$-connected.
For any $i, j\in [n]$ and $\epsilon, \delta \in \{0, 1\}$, let  $\overline{\{(i, \epsilon), (j, \delta)\}}$ be a simplicial complex on vertex set $\{(i, \epsilon), (j, \delta)\}$, which is isomorphic to $\bs^0$.
It is easy to check that 
\[\mathbf{N}(\{\X_n^{i, \epsilon}\}) \cong \overline{\{(1, 0), (1,1)\} } \ast \overline{\{(2, 0), (2,1)\} }  \ast \ldots \ast  \overline{ \{ (n, 0), (n,1)\}} \]
is the join of $n$ copies of $\bs^{0}$.
Hence $\mathbf{N}(\{\X_n^{i, \epsilon}\}) \simeq \bs^{n-1}$.
Since $n \geq 4$, we see that $\mathbf{N}(\{\X_n^{i, \epsilon}\})$ is $2$-connected.
\end{proof}

\begin{theorem}
\label{thm:main3}
Let $n \geq 4$.
Then $\tilde{H}_i(\N(\I_n, 3);\Z) \neq 0$ if and only if $i \in \{3, 4\}$.
\end{theorem}

\begin{proof}
From \Cref{lem:upto3}, we have $\tilde{H}_i(\N(\I_n, 3);\Z) = 0$ for $0 \leq i \leq 2$.
Since $\tilde{H}_i(\N(\I_4, 3);\Z) \neq 0$ for $i \in \{3,4\}$ from a computer computation (\Cref{table}), using the retractions in \Cref{lem:retraction} we get that $\tilde{H}_i(\N(\I_n, 3);\Z) \neq 0$ for $i \in \{3, 4\}$.
\Cref{thm:collapsibility} and \Cref{prop:collapsibilitysubcomplex} imply that 
$\tilde{H}_i(\N(\I_n, 3);\Z) = 0$ for $i \geq 5$.
\end{proof}

\section{Persistence}
\label{sec:persistence}

In this section, we show that in the persistent (reduced) homology of \v{C}ech complexes of hypercube graphs, no bars in the persistence barcode have length longer than two filtration steps.
Indeed, we will show that for any integers $n$ and $r$, the inclusion $\N(\I_n, r) \hookrightarrow \N(\I_n, r+2)$ is null-homotopic.
The proof technique was shown to us by \v{Z}iga Virk, who recently proved that Vietoris--Rips complexes of hypercube graphs have no persistence intervals of length longer than one filtration step (see~\cite{adams2023lower}).

\begin{theorem}
\label{thm:persistence}
Let $n \geq 1$ and $r \geq 0$.
Then the inclusion $\iota \colon \N(\I_n, r) \hookrightarrow \N(\I_n, r+2)$ is homotopic to a constant map.
\end{theorem}

\begin{proof}
For $1\le i\le n$, let $p_i\colon \{0,1\}^n \to \{0,1\}^n$ be the map defined by
\[p_i(x_1\ldots x_i x_{i+1} \ldots x_n)=(0\ldots 0 x_{i+1} \ldots x_n).\]
First, we claim that each $p_i$ induces a well-defined map $\tilde{p}_i \colon \N(\I_n, r) \to \N(\I_n, r+2)$ defined on vertices by $\tilde{p}_i(v)=p_i(v)$, and extended linearly to simplices via $\tilde{p}_i(\{v_0,\ldots,v_k\})=\{p_i(v_0),\ldots,p_i(v_k)\}$.
To see that $\tilde{p}_i$ is well-defined, note that $\tilde{p}_i(N_r[v])\subseteq N_r[p_i(v)]$ for all $v\in V(G)$.
(We remark that $\tilde{p}_i$ would still be well-defined even if its codomain were the smaller simplicial complex $\N(\I_n, r+1)$.)
Define $\tilde{p}_0 \colon \N(\I_n, r) \to \N(\I_n, r+2)$ to be the inclusion map $\tilde{p}_0=\iota$.

Next, we show that for all $1\le i\le n$, the two maps $\tilde{p}_{i-1}, \tilde{p}_i \colon \N(\I_n, r) \to \N(\I_n, r+2)$ are contiguous.
Two simplicial maps $f,g\colon K\to L$ are \emph{contiguous} if for each simplex $\sigma\in K$, the union $f(\sigma)\cup g(\sigma)$ is a simplex in $L$.
When $f$ and $g$ are contiguous, they induce homotopy equivalent maps on geometric realizations.
We split the verification that $\tilde{p}_{i-1}$ and $\tilde{p}_i$ are contiguous into two cases.
When $r$ is even, for any vertex $v\in V$, we have both
\begin{itemize}
\item $\tilde{p}_i(N_{\frac{r}{2}}[v])\subseteq N_{\frac{r}{2}}[p_i(v)]\subseteq N_{\frac{r+2}{2}}[p_{i-1}(v)]$ since $d(p_{i-1}(v),p_i(v))\le 1$, and
\item $\tilde{p}_{i-1}(N_{\frac{r}{2}}[v])\subseteq N_r[p_{i-1}(v)] \subseteq N_{\frac{r+2}{2}}[p_{i-1}(v)]$.
\end{itemize}
This shows that the maps $\tilde{p}_{i-1},\tilde{p}_i\colon\N(\I_n, r) \to \N(\I_n, r+2)$ are contiguous when $r$ is even.
Similarly, when $r$ is odd, for any edge $(v,w)\in E(G)$, we have both
\begin{itemize}
\item $\tilde{p}_i(N_{\frac{r-1}{2}}[v]\cup N_{\frac{r-1}{2}}[w])\subseteq N_{\frac{r-1}{2}}[p_i(v)]\cup N_{\frac{r-1}{2}}[p_i(w)]\subseteq N_{\frac{r+1}{2}}[p_{i-1}(v)]\cup N_{\frac{r+1}{2}}[p_{i-1}(w)]$ since $d(p_{i-1}(v),p_i(v))\le 1$ and $d(p_{i-1}(w),p_i(w))\le 1$, and
\item $\tilde{p}_{i-1}(N_{\frac{r-1}{2}}[v]\cup N_{\frac{r-1}{2}}[w])\subseteq N_{\frac{r-1}{2}}[p_{i-1}(v)]\cup N_{\frac{r-1}{2}}[p_{i-1}(w)] \subseteq N_{\frac{r+1}{2}}[p_{i-1}(v)]\cup N_{\frac{r+1}{2}}[p_{i-1}(w)]$.
\end{itemize}
Hence the maps $\tilde{p}_{i-1},\tilde{p}_i\colon\N(\I_n, r) \to \N(\I_n, r+2)$ are contiguous.

So for all $1\le i\le n$, the maps $\tilde{p}_{i-1}, \tilde{p}_i$ are contiguous and hence homotopic.
We have the chain of homotopy equivalences $\iota = \tilde{p}_0 \simeq \tilde{p}_1 \simeq \ldots \simeq \tilde{p}_{i-1} \simeq \tilde{p}_n$.
Since $\tilde{p}_n$ is a constant map, this shows that the inclusion $\iota \colon \N(\I_n, r) \hookrightarrow \N(\I_n, r+2)$ is a null-homotopy, as claimed.
\end{proof}

\Cref{thm:persistence} implies that in the persistent (reduced) homology of \v{C}ech complexes of hypercube graphs, no bars in the persistence barcode have length longer than two filtration steps.

We remark that our proof strategy above will not work with the inclusion $\iota \colon \N(\I_n, r) \hookrightarrow \N(\I_n, r+1)$, when $r+2$ in the codomain is replaced by $r+1$.
Indeed, consider the  2-dimensional hypercube $\I_2$ with $n=2$, and let $r=1$ be odd.
The complex $\N(\I_2,1)$ consists of four vertices and four edges, arranged in a square.
Let $v=11$ and $w=10$.
Note that
\[ N_{\frac{r-1}{2}}[v]\cup N_{\frac{r-1}{2}}[w] = N_0[11]\cup N_0[10] = \{11,10\} \]
and so 
\[ \iota\left(N_{\frac{r-1}{2}}[v]\cup N_{\frac{r-1}{2}}[w]\right)\cup\tilde{p}_1\left(N_{\frac{r-1}{2}}[v]\cup N_{\frac{r-1}{2}}[w]\right)=\{11,10,01,00\}.\]
There is no vertex $v\in \I_2$ such that $N_{\frac{r+1}{2}}[v]=N_1[v]$ contains this set, as the complex $\N(\I_2,2)$ is the \emph{boundary} of a tetrahedron.
Therefore the maps $\iota,\tilde{p}_1 \colon \N(\I_n, r) \hookrightarrow \N(\I_n, r+1)$ are not contiguous.
We do not know if the inclusions $\iota \colon \N(\I_n, r) \hookrightarrow \N(\I_n, r+1)$ are always nullhomotopic or not; see \Cref{ques:persistence}.

\section{Conclusion and open questions}
\label{sec:conclusion}

We end with some open questions.
Though determining all of the homotopy types of \v{C}ech complexes of hypercubes ({\it i.e.}\ all of the homotopy types in \Cref{table}) may be a difficult task, we hope these open questions will inspire further progress.

Since $\N(\I_n,1)$ is isomorphic to the graph $\I_n$ itself, the complex $\N(\I_n, 1)$ is a wedge of circles.
Moreover, from \Cref{thm:general_for_index_2}, we know that $\N(\I_n,2)$ is also homotopy equivalent to a wedge of spheres.
Thus the following is a very natural question.
\begin{question}
For any $n\ge 1$ and $r\ge 0$, is the \v{C}ech complex $\N(\I_n,r)$ of the hypercube graph always homotopy equivalent to a wedge of spheres?
\end{question}

From \Cref{thm:general_for_index_2}, $\tilde{H}_i(\N(\I_n, 2)) \neq 0$ only for $i = 2$.
For $r= 3$, in  \Cref{thm:main3}, we showed that 
$\tilde{H}_i(\N(\I_n, r)) \neq 0$ only if $i \in \{3, 4\}$.
Therefore, in view of the \Cref{table}, we propose  the following question.

\begin{question}
When is $\tilde{H}_i(\N(\I_n, r))$ nonzero?
\begin{itemize}
\item For $r = 2k$ even and for $n\ge k+1$, is $\tilde{H}_i(\N(\I_n), r) \neq 0$ if and only if $i \in \{r, 2(2^{k}-1)\}$?
\item For  $r=2k+1$ odd and for $n\ge k+2$, is $\tilde{H}_i(\N(\I_n), r) \neq 0$ if and only if  $i \in \{r, 3\cdot 2^k-2 \}$?
\end{itemize}
\end{question}

Another interesting problem would be to find the collapsibility number of $\N(\I_n,r)$.
Since $\N(\I_n,1)$ is isomorphic to the graph $\I_n$ itself, the collapsibility number of $\N(\I_n, 1)$ is $2$.

\begin{lemma}
For $n \geq 2$, the collapsibility number of $\N(\I_n,2)$ is $3$.
\end{lemma}
\begin{proof}
It is easy to see that, for any two vertices $v,w \in \I_n$, $N_1[v]\cap N_1[w]$ has at most $2$ elements.
Therefore, for any two  maximal simplices $\tau, \sigma \in \N(\I_n,2)$, if $|\tau \cap \sigma | \geq 3$, then  $\sigma = \tau$.
Hence, by using \Cref{prop:minimalexclusion}, we conclude that  $\N(\I_n,2)$ is $3$-collapsible.
The result then follows from \Cref{thm:general_for_index_2} and \Cref{prop:collapsibilitysubcomplex}.
\end{proof}

From \Cref{thm:collapsibility}, we know that $\N(\I_n,3)$ is 5-collapsible.
In this direction, we ask the following.

\begin{question}
What is the collapsibility number of $\N(\I_n,r)$?
\begin{itemize}
\item For $r=2k$ even and for $n\ge k+1$, is the collapsibility number of $\N(\I_n,2k)$ equal to $2^{k+1}-1$? 
\item For $r=2k+1$ odd and for $n\ge k+2$, is the collapsibility number of $\N(\I_n,2k+1)$ equal to $3\cdot 2^k-1$?
\end{itemize}
\end{question}

\begin{question}
\label{ques:persistence}
In \Cref{thm:persistence}, we prove that the inclusion $\iota \colon \N(\I_n, r) \hookrightarrow \N(\I_n, r+2)$ is homotopic to a constant map.
Are the inclusions $\iota \colon \N(\I_n, r) \hookrightarrow \N(\I_n, r+1)$ also always null-homotopies?
\end{question}

\section*{Acknowledgments}

We would like to thank Micha{\l} Adamaszek and \v{Z}iga Virk for helpful conversations.
Henry Adams is supported by a Simons Foundation's progarm Travel Support for Mathematicians.
Anurag Singh is supported by the Start-up Research Grant SRG/2022/000314 from SERB, DST, India.


\bibliographystyle{abbrv}
\bibliography{CechComplexesOfHypercubes}

\begin{thebibliography}{10}

\bibitem{Adamaszek2013}
M.~Adamaszek.
\newblock Clique complexes and graph powers.
\newblock {\em Israel Journal of Mathematics}, 196(1):295--319, 2013.

\bibitem{AA-VRS1}
M.~Adamaszek and H.~Adams.
\newblock The {V}ietoris--{R}ips complexes of a circle.
\newblock {\em Pacific Journal of Mathematics}, 290:1--40, 2017.

\bibitem{adamaszek2022vietoris}
M.~Adamaszek and H.~Adams.
\newblock On {V}ietoris--{R}ips complexes of hypercube graphs.
\newblock {\em Journal of Applied and Computational Topology}, 6:177--192,
  2022.

\bibitem{AAFPP-J}
M.~Adamaszek, H.~Adams, F.~Frick, C.~Peterson, and C.~Previte-Johnson.
\newblock Nerve complexes of circular arcs.
\newblock {\em Discrete \& Computational Geometry}, 56:251--273, 2016.

\bibitem{adams2022geometric}
H.~Adams and B.~Coskunuzer.
\newblock Geometric approaches to persistent homology.
\newblock {\em Accepted to appear in SIAM Journal on Applied Algebra and
  Geometry}, 2022.

\bibitem{adams2023lower}
H.~Adams and {\v{Z}}.~Virk.
\newblock Lower bounds on the homology of {V}ietoris--{R}ips complexes of
  hypercube graphs.
\newblock {\em arXiv preprint arXiv:2309.06222}, 2023.

\bibitem{akkiraju1995alpha}
N.~Akkiraju, H.~Edelsbrunner, M.~Facello, P.~Fu, E.~Mucke, and C.~Varela.
\newblock Alpha shapes: definition and software.
\newblock In {\em Proceedings of the 1st international computational geometry
  software workshop}, volume~63, 1995.

\bibitem{Alexandroff1928}
P.~Alexandroff.
\newblock {\"U}ber den allgemeinen {D}imensionsbegriff und seine {B}eziehungen
  zur elementaren geometrischen {A}nschauung.
\newblock {\em Mathematische Annalen}, 98(1):617--635, 1928.

\bibitem{bauer2021ripser}
U.~Bauer.
\newblock Ripser: efficient computation of {V}ietoris--{R}ips persistence
  barcodes.
\newblock {\em Journal of Applied and Computational Topology}, pages 391--423,
  2021.

\bibitem{bjorner}
A.~Bj\"{o}rner.
\newblock Topological methods.
\newblock In {\em Handbook of combinatorics, {V}ol. 1, 2}, pages 1819--1872.
  Elsevier Sci. B. V., Amsterdam, 1995.

\bibitem{bjorner1995homology}
A.~Bj\"{o}rner and V.~Welker.
\newblock The homology of ``$k$-equal'' manifolds and related partition
  lattices.
\newblock {\em Advances in mathematics}, 110(2):277--313, 1995.

\bibitem{Borsuk1948}
K.~Borsuk.
\newblock On the imbedding of systems of compacta in simplicial complexes.
\newblock {\em Fundamenta Mathematicae}, 35(1):217--234, 1948.

\bibitem{camara2016topological}
P.~G. Camara, D.~I. Rosenbloom, K.~J. Emmett, A.~J. Levine, and R.~Rabad\'{a}n.
\newblock Topological data analysis generates high-resolution, genome-wide maps
  of human recombination.
\newblock {\em Cell systems}, 3(1):83--94, 2016.

\bibitem{Carlsson2009}
G.~Carlsson.
\newblock Topology and data.
\newblock {\em Bulletin of the American Mathematical Society}, 46(2):255--308,
  2009.

\bibitem{carlsson2020persistent}
G.~Carlsson and B.~Filippenko.
\newblock Persistent homology of the sum metric.
\newblock {\em Journal of Pure and Applied Algebra}, 224(5):106244, 2020.

\bibitem{chan2013topology}
J.~M. Chan, G.~Carlsson, and R.~Rabad\'{a}n.
\newblock Topology of viral evolution.
\newblock {\em Proceedings of the National Academy of Sciences},
  110(46):18566--18571, 2013.

\bibitem{chazal2009gromov}
F.~Chazal, D.~Cohen-Steiner, L.~J. Guibas, F.~M{\'e}moli, and S.~Y. Oudot.
\newblock Gromov--{H}ausdorff stable signatures for shapes using persistence.
\newblock In {\em Computer Graphics Forum}, volume~28, pages 1393--1403, 2009.

\bibitem{ChazalDeSilvaOudot2014}
F.~Chazal, V.~de~Silva, and S.~Oudot.
\newblock Persistence stability for geometric complexes.
\newblock {\em Geometriae Dedicata}, 174:193--214, 2014.

\bibitem{conant2010boolean}
J.~Conant and O.~Thistlethwaite.
\newblock Boolean formulae, hypergraphs and combinatorial topology.
\newblock {\em Topology and its Applications}, 157(16):2449--2461, 2010.

\bibitem{DeSilvaCarlsson}
V.~de~Silva and G.~E. Carlsson.
\newblock Topological estimation using witness complexes.
\newblock {\em Eurographics Symposium on Point-Based Graphics}, 4:157--166,
  2004.

\bibitem{EdelsbrunnerHarer}
H.~Edelsbrunner and J.~L. Harer.
\newblock {\em Computational Topology: An Introduction}.
\newblock American Mathematical Society, Providence, 2010.

\bibitem{edelsbrunner2000topological}
H.~Edelsbrunner, D.~Letscher, and A.~Zomorodian.
\newblock Topological persistence and simplification.
\newblock In {\em Proceedings 41st Symposium on Foundations of Computer
  Science}, pages 454--463. IEEE, 2000.

\bibitem{edelsbrunner1994three}
H.~Edelsbrunner and E.~P. M{\"u}cke.
\newblock Three-dimensional alpha shapes.
\newblock {\em ACM Transactions on Graphics (TOG)}, 13(1):43--72, 1994.

\bibitem{emmett2015quantifying}
K.~Emmett and R.~Rabad\'{a}n.
\newblock Quantifying reticulation in phylogenetic complexes using homology.
\newblock {\em arXiv preprint arXiv:1511.01429}, 2015.

\bibitem{emmett2016topology}
K.~J. Emmett.
\newblock {\em Topology of Reticulate Evolution}.
\newblock PhD thesis, Columbia University, 2016.

\bibitem{feng2023homotopy}
Z.~Feng.
\newblock Homotopy types of {V}ietoris--{R}ips complexes of hypercube graphs.
\newblock {\em arXiv preprint arXiv:2305.07084}, 2023.

\bibitem{feng2023vietoris}
Z.~Feng and N.~C.~P. Nukala.
\newblock On {V}ietoris--{R}ips complexes of finite metric spaces with scale
  $2$.
\newblock {\em arXiv preprint arXiv:2302.14664}, 2023.

\bibitem{gasparovic2018complete}
E.~Gasparovic, M.~Gommel, E.~Purvine, R.~Sazdanovic, B.~Wang, Y.~Wang, and
  L.~Ziegelmeier.
\newblock A complete characterization of the one-dimensional intrinsic \v{C}ech
  persistence diagrams for metric graphs.
\newblock In {\em Research in Computational Topology}, pages 33--56. Springer,
  2018.

\bibitem{polymake:2000}
E.~Gawrilow and M.~Joswig.
\newblock {\tt Polymake}: {A} framework for analyzing convex polytopes.
\newblock In {\em Polytopes -- Combinatorics and Computation ({O}berwolfach,
  1997)}, volume~29 of {\em DMV Sem.}, pages 43--73. Birkh\"auser, Basel, 2000.

\bibitem{SSA22}
S.~Goyal, S.~Shukla, and A.~Singh.
\newblock Topology of clique complexes of line graphs.
\newblock {\em Art Discrete Appl. Math.}, 5(2):Paper No. 2.06, 12, 2022.

\bibitem{guibas2008reconstruction}
L.~J. Guibas and S.~Y. Oudot.
\newblock Reconstruction using witness complexes.
\newblock {\em Discrete \& Computational geometry}, 40(3):325--356, 2008.

\bibitem{Hatcher}
A.~Hatcher.
\newblock {\em Algebraic Topology}.
\newblock Cambridge University Press, Cambridge, 2002.

\bibitem{Hausmann1995}
J.-C. Hausmann.
\newblock On the {V}ietoris--{R}ips complexes and a cohomology theory for
  metric spaces.
\newblock {\em Annals of Mathematics Studies}, 138:175--188, 1995.

\bibitem{Latschev2001}
J.~Latschev.
\newblock Vietoris--{R}ips complexes of metric spaces near a closed
  {R}iemannian manifold.
\newblock {\em Archiv der Mathematik}, 77(6):522--528, 2001.

\bibitem{Lew2018}
A.~Lew.
\newblock Collapsibility of simplicial complexes of hypergraphs.
\newblock {\em Electron. J. Combin.}, 26(4):Paper No. 4.10, 10, 2019.

\bibitem{lim2020vietoris}
S.~Lim, F.~M{\'e}moli, and O.~B. Okutan.
\newblock Vietoris--{R}ips persistent homology, injective metric spaces, and
  the filling radius.
\newblock {\em arXiv preprint arXiv:2001.07588}, 2020.

\bibitem{lovasz}
L.~Lov\'{a}sz.
\newblock Kneser's conjecture, chromatic number, and homotopy.
\newblock {\em J. Combin. Theory Ser. A}, 25(3):319--324, 1978.

\bibitem{niyogi2008finding}
P.~Niyogi, S.~Smale, and S.~Weinberger.
\newblock Finding the homology of submanifolds with high confidence from random
  samples.
\newblock {\em Discrete \& Computational Geometry}, 39(1):419--441, 2008.

\bibitem{Previte2014}
C.~Previte-Johnson.
\newblock {\em The $D$-Neighborhood Complex of a Graph}.
\newblock PhD thesis, Colorado State University, 2014.

\bibitem{shukla2022vietoris}
S.~Shukla.
\newblock On {V}ietoris--{R}ips complexes (with scale 3) of hypercube graphs.
\newblock {\em SIAM Journal on Discrete Mathematics}, 37(3):1472--1495, 2023.

\bibitem{oeis}
N.~J.~A. Sloane.
\newblock The {O}n-{L}ine {E}ncyclopedia of {I}nteger {S}equences.
  \href{http://oeis.org}{http://oeis.org}.

\bibitem{Vietoris27}
L.~Vietoris.
\newblock {{\"U}ber den h{\"o}heren Zusammenhang kompakter R{\"a}ume und eine
  Klasse von zusammenhangstreuen Abbildungen}.
\newblock {\em Mathematische Annalen}, 97(1):454--472, 1927.

\bibitem{virk20201}
{\v{Z}}.~Virk.
\newblock 1-dimensional intrinsic persistence of geodesic spaces.
\newblock {\em Journal of Topology and Analysis}, 12:169--207, 2020.

\bibitem{Wegner1975}
G.~Wegner.
\newblock {$d$}-{C}ollapsing and nerves of families of convex sets.
\newblock {\em Arch. Math. (Basel)}, 26:317--321, 1975.

\bibitem{zomorodian2005computing}
A.~Zomorodian and G.~Carlsson.
\newblock Computing persistent homology.
\newblock {\em Discrete \& Computational Geometry}, 33(2):249--274, 2005.

\end{thebibliography}

\end{document}